\documentclass[12pt,reqno]{amsart}
\usepackage{amsmath, amsthm, amssymb}

\advance \topmargin by -\headheight
\advance \topmargin by -\headsep
     
\evensidemargin \oddsidemargin
\newtheorem{theorem}{Theorem}[section]
\newtheorem{lemma}[theorem]{Lemma}
\newtheorem{corollary}[theorem]{Corollary}

\theoremstyle{definition}

\theoremstyle{remark}

\numberwithin{equation}{section}

\newcommand{\mmod}[1]{\,\,(\text{mod}\,\,#1)}

\def\bfh{{\mathbf h}}

\def\bfm{{\mathbf m}}
\def\bfn{{\mathbf n}}

\def\bfu{{\mathbf u}}
\def\bfv{{\mathbf v}}
\def\bfw{{\mathbf w}}
\def\bfx{{\mathbf x}}
\def\bfy{{\mathbf y}}
\def\bfz{{\mathbf z}}

\def\calB{{\mathcal B}} 
\def\calC{{\mathcal C}} 

\def\calD{{\mathcal D}}

\def\calJ{{\mathcal J}}

 \def\Ktil{{\widetilde K}}

\def\calP{{\mathcal P}}

\def\calR{{\mathcal R}}

\def\util{\tilde{u}}\def\bfutil{\mathbf \util}\def\vtil{\tilde{v}}

\def\gtil{\tilde{g}}

\def\atil{\tilde{a}}\def\btil{\tilde{b}}

\def\Gtil{\widetilde G}\def\Itil{\widetilde I}\def\Ktil{\widetilde K}
\def\Atil{\tilde A}\def\Btil{\tilde B}

\def\dbC{{\mathbb C}}\def\dbN{{\mathbb N}}
\def\dbR{{\mathbb R}}
\def\dbZ{{\mathbb Z}}

\def\gra{{\mathfrak a}}\def\grA{{\mathfrak A}}
\def\grb{{\mathfrak b}}\def\grB{{\mathfrak B}}
\def\grC{{\mathfrak C}}

\def\grf{{\mathfrak f}}\def\grF{{\mathfrak F}}
\def\grG{{\mathfrak G}}
\def\grH{{\mathfrak H}}

\def\grS{{\mathfrak S}}
\def\grB{{\mathfrak B}}\def\grC{{\mathfrak C}}

\def\grs{{\mathfrak s}}

\def\alp{{\alpha}} \def\bfalp{{\boldsymbol \alpha}}
\def\bet{{\beta}}  
\def\gam{{\gamma}} \def\Gam{{\Gamma}}
\def\del{{\delta}} \def\Del{{\Delta}} \def\Deltil{{\widetilde \Delta}}
\def\zet{{\zeta}} \def\bfzet{{\boldsymbol \zeta}} 
\def\bfeta{{\boldsymbol \eta}} 
\def\tet{{\theta}} \def\bftet{{\boldsymbol \theta}} \def\Tet{{\Theta}}
\def\kap{{\kappa}}
\def\lam{{\lambda}} \def\Lam{{\Lambda}} 

\def\bfxi{{\boldsymbol \xi}}

\def\sig{{\sigma}}  

\def\Ups{{\Upsilon}} 
\def\bfpsi{{\boldsymbol \psi}}
\def\ome{{\omega}} \def\Ome{{\Omega}}
\def\d{{\partial}}
\def\eps{\varepsilon}

\def\le{\leqslant} \def\ge{\geqslant}

\def\d{{\,{\rm d}}}

\def\llbracket{\lbrack\;\!\!\lbrack} \def\rrbracket{\rbrack\;\!\!\rbrack}

\begin{document}
\title[Approximating the main conjecture]{Approximating the main conjecture in Vinogradov's mean value theorem}
\author[Trevor D. Wooley]{Trevor D. Wooley}
\address{School of Mathematics, University of Bristol, University Walk, Clifton, Bristol BS8 1TW, United 
Kingdom}
\email{matdw@bristol.ac.uk}
\subjclass[2010]{11L15, 11L07, 11P55}
\keywords{Exponential sums, Hardy-Littlewood method}
\date{}
\begin{abstract} We apply multigrade efficient congruencing to estimate Vinogradov's integral of degree 
$k$ for moments of order $2s$, establishing strongly diagonal behaviour for 
$1\le s\le \frac{1}{2}k(k+1)-\frac{1}{3}k+o(k)$. In particular, as $k\rightarrow \infty$, we confirm the 
main conjecture in Vinogradov's mean value theorem for 100\% of the critical interval 
$1\le s\le \frac{1}{2}k(k+1)$.\end{abstract}
\maketitle

\section{Introduction} Our focus in this paper is the Vinogradov system of Diophantine equations. When 
$k$ and $s$ are natural numbers, and $X$ is a large real number, denote by $J_{s,k}(X)$ the number of
 integral solutions of the system
\begin{equation}\label{1.1}
x_1^j+\ldots +x_s^j=y_1^j+\ldots +y_s^j\quad (1\le j\le k),
\end{equation}
with $1\le x_i,y_i\le X$ $(1\le i\le s)$. By considering mean values of an associated exponential sum, it 
follows that the validity of the {\it strongly diagonal} estimate
\begin{equation}\label{1.2}
J_{s,k}(X)\ll X^{s+\eps},
\end{equation}
for $1\le s\le \frac{1}{2}k(k+1)$, would imply the {\it main conjecture} in Vinogradov's mean value 
theorem, which asserts that for each $\eps>0$, one has
\begin{equation}\label{1.3}
J_{s,k}(X)\ll X^\eps (X^s+X^{2s-\frac{1}{2}k(k+1)}).
\end{equation}
Here and throughout this paper, the implicit constants associated with Vinogradov's notation $\ll $ and
 $\gg $ may depend on $s$, $k$ and $\eps$. For almost all of the eighty year history of the subject, such 
an estimate has seemed a very remote prospect indeed, for until only a year ago the bound (\ref{1.2}) 
was known to hold only for $s\le k+1$ (see \cite[Lemma 5.4]{Hua1965}). By enhancing the {\it efficient 
congruencing} method introduced in \cite{Woo2012}, recent work of the author joint with Ford 
\cite{FW2013} has established the estimate (\ref{1.2}) for $s\le \frac{1}{4}(k+1)^2$, amounting to half of 
the range $s\le \frac{1}{2}k(k+1)$ in which diagonal behaviour dominates (\ref{1.3}). Our goal in this 
paper is to adapt the {\it multigrade efficient congruencing method} detailed in our very recent work 
\cite{Woo2014} so as to obtain strongly diagonal estimates of the shape (\ref{1.2}). It transpires that for 
large $k$, we are able to establish the validity of the conjectured estimate (\ref{1.2}) in 100\% of the 
critical interval $1\le s\le \frac{1}{2}k(k+1)$. We consequently find ourselves within a whisker of the proof 
of the main conjecture (\ref{1.3}), a situation unprecedented in the analysis of mean value estimates for 
exponential sums of large degree.\par

The most precise form of our new estimate may be found in Theorem \ref{theorem9.2} below. For the 
present, we restrict ourselves to a conclusion relatively simple to state that is indicative of what may now be 
achieved.

\begin{theorem}\label{theorem1.1}
Let $k$ be a natural number with $k\ge 7$, put $r=k-\lceil 2\sqrt{k}\rceil +2$, and suppose that $s$ is 
an integer satisfying
$$1\le s\le kr-\tfrac{1}{2}r(r-1)-\sum_{m=1}^{r-1}\frac{m(r-m)}{s-r-m}.$$
Then for each $\eps>0$, one has $J_{s,k}(X)\ll X^{s+\eps}$.
\end{theorem}

A modicum of computation reveals that when $k$ is large, this estimate confirms the main conjecture 
(\ref{1.3}) for almost all of the critical interval of exponents $1\le s\le \tfrac{1}{2}k(k+1)$.

\begin{corollary}\label{corollary1.2}
When $s$ and $k$ are natural numbers with $k\ge 7$ and
$$1\le s\le \tfrac{1}{2}k(k+1)-\tfrac{7}{3}k,$$ 
then for each $\eps>0$, one has $J_{s,k}(X)\ll X^{s+\eps}$.
\end{corollary}

As we have already noted, a conclusion analogous to Theorem \ref{theorem1.1} is made available in 
\cite[Theorem 1.1]{FW2013} in the more limited range $1\le s\le \frac{1}{4}(k+1)^2$. Earlier 
conclusions were limited to the much shorter interval $1\le s\le k+1$ (see \cite{VW1997} for a particularly 
precise statement in this situation). In view of the lower bound
\begin{equation}\label{1.4}
J_{s,k}(X)\gg X^s+X^{2s-\frac{1}{2}k(k+1)},
\end{equation}
that arises by considering the diagonal solutions of (\ref{1.1}) with $\bfx=\bfy$, together with a lower 
bound for the product of local densities (see \cite[equation (7.4)]{Vau1997}), one sees that the 
conclusion of Theorem \ref{theorem1.1} cannot hold when $s>\tfrac{1}{2}k(k+1)$.\par

The gap between our new result and a complete proof of the main conjecture (\ref{1.3}) stands at 
$\tfrac{7}{3}k$ variables, amounting to a proportion $O(1/k)$ of the length of the critical 
interval $1\le s\le \frac{1}{2}k(k+1)$. In a certain sense, therefore, our conclusion establishes the main 
conjecture in 100\% of the critical interval as $k\rightarrow \infty$, justifying the assertion concluding our 
opening paragraph. In \S11 we describe some modifications to our basic strategy which offer some 
improvement in the conclusion of Corollary \ref{corollary1.2} towards the main conjecture (\ref{1.3}). 
We have opted for a relatively concise account of this improvement, since the details of the associated 
argument are of sufficient complexity that the key elements of our basic strategy would be obscured were 
we to make this account the main focus of our exposition.\par

\begin{theorem}\label{theorem1.3}
When $k$ is a sufficiently large natural number, and $s$ is an integer satisfying 
$$1\le s\le \tfrac{1}{2}k(k+1)-\tfrac{1}{3}k-8k^{2/3},$$
then for each $\eps>0$, one has $J_{s,k}(X)\ll X^{s+\eps}$.
\end{theorem}

Our methods are applicable not only for large values of $k$. By taking $r=k-\lceil 2\sqrt{k}\rceil +2$ in 
Theorem \ref{theorem9.2} below, one may establish an improvement on the conclusion of 
\cite[Theorem 1.1]{FW2013}. We note here that the latter delivers the bound $J_{s,k}(X)\ll X^{s+\eps}$ 
for $1\le s\le \frac{1}{4}(k+1)^2$.\par

\begin{theorem}\label{theorem1.4}
Define the exponent $D(k)$ as in Table 1. Then whenever $k\ge 4$ and $1\le s\le D(k)$, one has 
$J_{s,k}(X)\ll X^{s+\eps}$.
\end{theorem}

$$\boxed{\begin{matrix} k&4&5&6&7&8&9&10&11&12\\
D(k)&8&10&17&20&29&38&44&55&68\\
\tfrac{1}{2}k(k+1)&10&15&21&28&36&45&55&66&78\end{matrix}}$$
$$\boxed{\begin{matrix}k&13&14&15&16&17&18&19&20\\
D(k)&75&90&105&122&132&151&170&191\\
\tfrac{1}{2}k(k+1)&91&105&120&136&153&171&190&210\end{matrix}}$$
\vskip.2cm
\begin{center}\text{Table 1: Values for $D(k)$ described in Theorem \ref{theorem1.4}.}\end{center}
\vskip.1cm
\noindent 

One may enquire concerning the impact of our methods in the case $k=3$. Here, although the methods of 
this paper do not apply directly, a careful variant of our ideas proves surprisingly successful. This is a 
matter which we address elsewhere.\par

A trivial consequence of the estimate supplied by Theorems \ref{theorem1.1} and \ref{theorem1.3} 
provides a new upper bound for $J_{s,k}(X)$ at the critical value $s=\frac{1}{2}k(k+1)$.

\begin{theorem}\label{theorem1.5}
Suppose that $k$ is a large natural number. Then for each $\eps>0$, one has
$$J_{\frac{1}{2}k(k+1),k}(X)\ll X^{\frac{1}{2}k(k+1)+\Del+\eps},$$
where $\Del=(\tfrac{1}{3}+o(1))k$.
\end{theorem}

As remarked in the discussion following the statement of \cite[Theorem 1.2]{FW2013}, the work of that 
paper shows that a conclusion analogous to Theorem \ref{theorem1.5} holds with 
$\Del=(\frac{3}{2}-\sqrt{2})k^2+O(k)$. Our new result for the first time in the subject obtains a 
conclusion in which $\Del=o(k^2)$. Indeed, one now has a good asymptotic approximation to the main 
conjecture (\ref{1.3}) for all $s$, since we have
$$J_{s,k}(X)\ll X^{\Del_{s,k}}(X^s+X^{2s-\frac{1}{2}k(k+1)}),$$
with $\Del_{s,k}=O(k)$ for all $s$.\par

We are also able to obtain a very slight improvement in our previous bound for the the number of 
variables required to establish the anticipated asymptotic formula in Waring's problem, at least, when $k$ 
is sufficiently large. We refer the reader to Theorem \ref{theorem12.2} for an account of this new bound, 
and to Theorem \ref{theorem12.1} for some consequences of Theorem \ref{theorem1.3} in the context 
of Tarry's problem. Other applications of Vinogradov's mean value theorem can be found in the classical 
literature on the subject, such as \cite{ACK2004, Vin1947, Woo1992}.\par

We have written elsewhere concerning the basic strategy underlying efficient congruencing; a sketch of the 
basic ideas can be found in \cite[\S2]{Woo2012}. The multigrade efficient congruencing method seeks to 
more fully extract the information available from congruences of different levels. We direct the reader to 
\cite{Woo2014} for an introduction to such ideas. Perhaps it is worth noting that the main challenge in the 
present adaptation of multigrade efficient congruencing is that no loss of more than a constant factor can 
be tolerated in the basic congruencing steps. This task is complicated as we progress through the iteration, 
stepping from congruences modulo $p^a$ to congruences modulo $p^b$, since the ratio $b/a$ can vary 
widely according to which of the several possible moduli dominate proceedings at each stage of the 
iteration. Finally, we remark that although one can envision further improvement in the range of $s$ 
accommodated by Theorem \ref{theorem1.3}, it would seem that a new idea is required to replace the 
defect $\tfrac{1}{3}k$ by a quantity $ck$ with $c<\tfrac{1}{3}$. The term $\tfrac{1}{3}k$ seems to be 
an unavoidable consequence of the application of H\"older's inequality underlying Lemmata 
\ref{lemma7.1} and \ref{lemma11.3} that teases apart the information available from congruences of 
different levels.

\section{The basic infrastructure} We begin our discussion of the the proof of Theorem \ref{theorem1.1} 
by describing the components and basic notation that we subsequently assemble into the apparatus 
required for the multigrade efficient congruencing method. This is very similar to that introduced in our 
recent work \cite{Woo2014}, and resembles the infrastructure of our earliest work on efficient 
congruencing \cite{Woo2012}. The reader might wonder whether some kind of universal account could be 
given that would encompass all possibilities. For the present such a goal seems premature, since new 
innovations are the rule for such a rapidly evolving circle of ideas. Instead, novel consequences of the 
method seem to require careful arrangement of parameters together with a significant measure of artistry.\par

We consider $k$ to be fixed, and abbreviate $J_{s,k}(X)$ to $J_s(X)$ without further comment. Let 
$s\in \dbN$ be arbitrary, and define $\lam_s^*\in \dbR$ by means of the relation
$$\lam_s^*=\underset{X\rightarrow \infty}{\lim \sup}\frac{\log J_s(X)}{\log X}.$$
Thus, for each $\eps>0$, and any $X\in \dbR$ sufficiently large in terms of $s$, $k$ and $\eps$, one has 
$J_s(X)\ll X^{\lam_s^*+\eps}$. Note that the lower bound (\ref{1.4}) combines with a trivial estimate for 
$J_s(X)$ to show that $s\le \lam_s^*\le 2s$, whilst the conjectured upper bound (\ref{1.2}) implies that
$\lam_s^*=s$ for $s\le \frac{1}{2}k(k+1)$.\par

Next we recall some standard notational conventions. The letters $s$ and $k$ denote natural numbers with 
$k\ge 3$, and $\eps$ denotes a sufficiently small positive number. Our basic parameter is $X$, a large real 
number depending at most on $k$, $s$ and $\eps$, unless otherwise indicated. Whenever $\eps$ appears 
in a statement, we assert that the statement holds for each $\eps>0$. As usual, we write 
$\lfloor\psi\rfloor$ to denote the largest integer no larger than $\psi$, and $\lceil \psi\rceil$ to denote the 
least integer no smaller than $\psi$. We make sweeping use of vector notation. Thus, with $t$ implied 
from the environment at hand, we write $\bfz\equiv \bfw\pmod{p}$ to denote that 
$z_i\equiv w_i\pmod{p}$ $(1\le i\le t)$, or $\bfz\equiv \xi\pmod{p}$ to denote that 
$z_i\equiv \xi\pmod{p}$ $(1\le i\le t)$, or $[\bfz\mmod{q}]$ to denote the $t$-tuple 
$(\zet_1,\ldots ,\zet_t)$, where for $1\le i\le t$ one has $1\le \zet_i\le q$ and $z_i\equiv \zet_i\mmod{q}$. 
Finally, we employ the convention that whenever $G:[0,1)^k\rightarrow \dbC$ is integrable, then
$$\oint G(\bfalp)\d\bfalp =\int_{[0,1)^k}G(\bfalp)\d\bfalp .$$
Thus, on writing
$$f_k(\bfalp;X)=\sum_{1\le x\le X}e(\alp_1x+\alp_2x^2+\ldots +\alp_kx^k),$$
where as usual $e(z)$ denotes $e^{2\pi iz}$, it follows from orthogonality that
\begin{equation}\label{2.1}
J_t(X)=\oint |f_k(\bfalp;X)|^{2t}\d\bfalp \quad (t\in \dbN).
\end{equation}

\par We next introduce the parameters which define the iterative method that is the central concern of this 
paper. We suppose throughout that $k\ge 4$, and put
\begin{equation}\label{2.2}
r_0=k-\lceil 2\sqrt{k}\rceil +2.
\end{equation}
We take $r$ to be an integral parameter satisfying $1\le r\le r_0$, and we consider a natural number $s$ 
with $k+1\le s\le \tfrac{1}{2}k(k+1)$. We then define $\Del=\Del(r,s)$ by taking
\begin{equation}\label{2.3}
\Del=\sum_{m=1}^{r-1}\frac{m(r-m)}{s-m},
\end{equation}
and suppose in what follows that
\begin{equation}\label{2.4}
s\ge \max\{ 2r-1,\tfrac{1}{2}r(r-1)+\Del \}.
\end{equation}
Next, we put
\begin{equation}\label{2.5}
\gra=kr-\tfrac{1}{2}r(r-1)-\Del\quad \text{and}\quad \grb=r\left( \tfrac{1}{2}r(r-1)+\Del\right) ,
\end{equation}
and then define
\begin{equation}\label{2.6}
\tet_+=\tfrac{1}{2}(\gra+\sqrt{\gra^2-4\grb})\quad \text{and}\quad 
\tet_-=\tfrac{1}{2}(\gra-\sqrt{\gra^2-4\grb}).
\end{equation}
We will suppose throughout that $s<\tet_+$. Here, some explanation may defuse confusion that may arise 
from the implicit dependence of $\tet_+$ on $s$. In practice, the definition (\ref{2.3}) ensures that $\Del$ 
is no larger than about $\tfrac{1}{3}k$, and hence one finds from (\ref{2.6}) that $\tet_+$ is roughly 
$$kr-\tfrac{1}{2}r(r+1)-\Del>\tfrac{1}{2}r(r-1)+\Del.$$
The condition $s<\tet_+$ is consequently easily verified, since there are relatively few values of $s$ to check, 
and in particular none exceeding $\tfrac{1}{2}k(k+1)$. Likewise, the condition (\ref{2.4}) is also easily 
verified for the values of $s$ open to discussion.\par

\par Our goal is to establish that $\lam^*_{s+r}=s+r$. Having established the latter, one finds by 
applying H\"older's inequality to the right hand side of (\ref{2.1}) that whenever $1\le t\le s+r$, one has
$$J_t(X)\le \Bigl( \oint |f_k(\bfalp;X)|^{2s+2r}\d\bfalp \Bigr)^{t/(s+r)}\ll X^{t+\eps},$$
whence $\lam_t^*=t$. Thus it is that the main conclusions of this paper follow by fucusing on 
$\lam_{s+r}^*$. Henceforth, for brevity we write $\lam=\lam^*_{s+r}$. \par

\par Let $R$ be a natural number sufficiently large in terms of $s$ and $k$. Specifically, we choose $R$ as 
follows. From (\ref{2.5}) and (\ref{2.6}), one finds that $\tet_+\tet_-=\grb$, and so the hypothesis 
$s<\tet_+$ ensures that
$$\tet_-=\grb/\tet_+<\grb/s\le \grb/(\tfrac{1}{2}r(r-1)+\Del)=r.$$
Observe next that $\tet_+>\tet_-$, so that on writing $\tau=\tet_+/\tet_-$, one finds that
\begin{align*}
\frac{1}{\tet_-^R}\Bigl( \frac{\tet_+^{R+1}-\tet_-^{R+1}}{\tet_+-\tet_-}-\frac{\tet_+\tet_-}{r\sqrt{k}}
\Bigl( \frac{\tet_+^R-\tet_-^R}{\tet_+-\tet_-}\Bigr)\Bigr)&=\frac{\tau^{R+1}-1}{\tau-1}-\frac{\tau \tet_-}{r\sqrt{k}}
\Bigl( \frac{\tau^R-1}{\tau-1}\Bigr) \\
&> \Bigl( 1-\frac{1}{\sqrt{k}}\Bigr)\Bigl( \frac{\tau^{R+1}-1}{\tau-1}\Bigr) \\
&>(1-1/\sqrt{k})\tau^R.
\end{align*}
Hence we obtain the upper bound
$$(1-1/\sqrt{k})\tet_+^R<\frac{\tet_+^{R+1}-\tet_-^{R+1}}{\tet_+-\tet_-}-\frac{\tet_+\tet_-}{r\sqrt{k}}
\Bigl( \frac{\tet_+^R-\tet_-^R}{\tet_+-\tet_-}\Bigr).$$
Recall again our assumption that $s<\tet_+$, and put $\nu=\tet_+-s>0$. Then we have
$$s^n=\tet_+^n(1-\nu/\tet_+)^n\le \tet_+^ne^{-\nu n/\tet_+}.$$
Consequently, if we take $R=\lceil W\tet_+/\nu\rceil$, with $W$ a large enough integer, then we ensure 
that
$$s^R\le e^{-W}\tet_+^R<(1-1/\sqrt{k})\tet_+^R,$$
whence
\begin{equation}\label{2.7}
s^R<\frac{\tet_+^{R+1}-\tet_-^{R+1}}{\tet_+-\tet_-}-\frac{\tet_+\tet_-}{r\sqrt{k}}\left( 
\frac{\tet_+^R-\tet_-^R}{\tet_+-\tet_-}\right).
\end{equation}
The significance of this condition will become apparent in due course (see equation (\ref{8.1}) 
below, and the ensuing discussion).\par

Having fixed $R$ satisfying these conditions, we take $N$ to be a natural number sufficiently 
large in terms of $s$, $k$ and $R$, and then put
\begin{equation}\label{2.8}
B=Nk^N,\quad \tet=(17N^2(s+k))^{-3RN},\quad \del=(1000N^4(s+k))^{-3RN}\tet. 
\end{equation}
In view of the definition of $\lam$, there exists a sequence of natural numbers $(X_l)_{l=1}^\infty$, 
tending to infinity, with the property that $J_{s+r}(X_l)>X_l^{\lam-\del}$ $(l\in \dbN)$. Also, 
provided that $X_l$ is sufficiently large, one has the corresponding upper bound 
$J_{s+r}(Y)<Y^{\lam+\del}$ for $Y\ge X_l^{1/2}$. We consider a fixed element $X=X_l$ of the 
sequence $(X_l)_{l=1}^\infty$, which we may assume to be sufficiently large in terms of $s$, $k$ 
and $N$. We put $M=X^\tet$, and note from (\ref{2.8}) that $X^\del<M^{1/N}$. Throughout, 
implicit constants may depend on $s$, $k$, $N$, and also on $\eps$ in view of our earlier 
convention, but not on any other variable.\par

We next introduce the cast of exponential sums and mean values appearing in our arguments. Let $p$ be a 
prime number with $M<p\le 2M$ to be fixed in due course. When $c$ and $\xi$ are non-negative integers, 
and $\bfalp \in [0,1)^k$, we define
\begin{equation}\label{2.9}
\grf_c(\bfalp;\xi)=\sum_{\substack{1\le x\le X\\ x\equiv \xi\mmod{p^c}}}e(\alp_1x+\alp_2x^2+\ldots 
+\alp_kx^k).
\end{equation}
When $1\le m\le k-1$, denote by $\Xi_c^m(\xi)$ the set of integral $m$-tuples $(\xi_1,\ldots ,\xi_m)$, 
with $1\le \bfxi\le p^{c+1}$ and $\bfxi\equiv \xi\pmod{p^c}$, and satisfying the property that 
$\xi_i\not \equiv \xi_j\pmod{p^{c+1}}$ for $i\ne j$. We then put
\begin{equation}\label{2.10}
\grF_c^m(\bfalp;\xi)=\sum_{\bfxi\in \Xi_c^m(\xi)}\prod_{i=1}^m\grf_{c+1}(\bfalp;\xi_i).
\end{equation}
When $a$ and $b$ are positive integers, we define
\begin{align}
I_{a,b}^{m,r}(X;\xi,\eta)&=\oint |\grF_a^m(\bfalp;\xi)^2\grf_b(\bfalp;\eta)^{2s+2r-2m}|\d\bfalp ,
\label{2.11}\\
K_{a,b}^{m,r}(X;\xi,\eta)&=\oint |\grF_a^m(\bfalp;\xi)^2\grF_b^r(\bfalp;\eta)^2
\grf_b(\bfalp;\eta)^{2s-2m}|\d \bfalp .\label{2.12}
\end{align}

\par Note that, by orthogonality, the mean value $I_{a,b}^{m,r}(X;\xi,\eta)$ counts the number of 
integral solutions of the system
\begin{equation}\label{2.13}
\sum_{i=1}^m(x_i^j-y_i^j)=\sum_{l=1}^{s+r-m}(v_l^j-w_l^j)\quad (1\le j\le k),
\end{equation}
with $1\le \bfx,\bfy,\bfv,\bfw\le X$, $\bfv\equiv \bfw\equiv \eta \mmod{p^b}$, and
$\bfx,\bfy \in \Xi_a^m(\xi)\mmod{p^{a+1}}$. Similarly, the mean value $K_{a,b}^{m,r}(X;\xi,\eta)$ 
counts the number of solutions of
\begin{equation}\label{2.14}
\sum_{i=1}^m(x_i^j-y_i^j)=\sum_{l=1}^r(u_l^j-v_l^j)+\sum_{n=1}^{s-m}(w_n^j-z_n^j)
\quad (1\le j\le k),
\end{equation}
with $1\le \bfx,\bfy,\bfu,\bfv,\bfw,\bfz\le X$, and satisfying $\bfw\equiv \bfz\equiv \eta\mmod{p^b}$,
$$\bfx,\bfy \in \Xi_a^m(\xi)\mmod{p^{a+1}}\quad \text{and}\quad \bfu,\bfv\in 
\Xi_b^r(\eta)\mmod{p^{b+1}}.$$
Given any one such solution of the system (\ref{2.14}), an application of the Binomial Theorem shows 
that $\bfx-\eta$, $\bfy-\eta$, $\bfu-\eta$, $\bfv-\eta$, $\bfw-\eta$, $\bfz-\eta$ is also a solution. Since in 
any solution counted by $K_{a,b}^{m,r}(X;\xi,\eta)$, one has $\bfu\equiv \bfv\equiv \eta \mmod{p^b}$ 
and $\bfw\equiv \bfz\equiv \eta \mmod{p^b}$, we deduce that
\begin{equation}\label{2.15}
\sum_{i=1}^m(x_i-\eta)^j\equiv \sum_{i=1}^m(y_i-\eta)^j\mmod{p^{jb}}\quad (1\le j\le k).
\end{equation}

\par We put
\begin{align}
I_{a,b}^{m,r}(X)&=\max_{1\le \xi\le p^a}
\max_{\substack{1\le \eta\le p^b\\ \eta\not\equiv \xi\mmod{p}}}I_{a,b}^{m,r}(X;\xi,\eta),\label{2.16}\\
K_{a,b}^{m,r}(X)&=\max_{1\le \xi\le p^a}
\max_{\substack{1\le \eta\le p^b\\ \eta\not\equiv \xi\mmod{p}}}K_{a,b}^{m,r}(X;\xi,\eta).\label{2.17}
\end{align}
The implicit dependence on $p$ in the above notation will be rendered irrelevant in \S5, since we fix the 
choice of this prime following the proof of Lemma \ref{lemma5.1}.\par

We must align the definition of $K_{0,b}^{m,r}(X)$ with the conditioning idea. When $\xi$ is an integer 
and $\bfzet$ is a tuple of integers, we denote by $\Xi^m(\bfzet)$ the set of $m$-tuples 
$(\xi_1,\ldots ,\xi_m)\in \Xi_0^m(0)$ such that $\xi_i\not\equiv \zet_j\mmod{p}$ for all $i$ and $j$. 
Recalling (\ref{2.9}), we put
$$\grF^m(\bfalp;\bfzet)=\sum_{\bfxi \in \Xi^m(\bfzet)}\prod_{i=1}^m\grf_1(\bfalp;\xi_i),$$
and then define
\begin{align}
\Itil_c^{m,r}(X;\eta)&=\oint |\grF^m(\bfalp;\eta)^2\grf_c(\bfalp;\eta)^{2s+2r-2m}|\d\bfalp ,
\label{2.18}\\
\Ktil_c^{m,r}(X;\eta)&=\oint |\grF^m(\bfalp;\eta)^2\grF_c^r(\bfalp;\eta)^2\grf_c(\bfalp;\eta)^{2s-2m}|
\d\bfalp ,\label{2.19}\\
K_{0,c}^{m,r}(X)&=\max_{1\le \eta\le p^c}\Ktil_c^{m,r}(X;\eta).\label{2.20}
\end{align}

\par As in our earlier work, we make use of an operator that makes transparent the relationship between 
mean values and their anticipated magnitudes. For the purposes at hand, it is useful to reconfigure this 
normalisation so as to make visible the deviation in a mean value from strongly diagonal behaviour. Thus, 
in the present circumstances, we adopt the convention that
\begin{align}
\llbracket J_{s+r}(X)\rrbracket&=J_{s+r}(X)/X^{s+r},\label{2.21}\\
\llbracket I_{a,b}^{m,r}(X)\rrbracket&=\frac{I_{a,b}^{m,r}(X)}{(X/M^a)^m(X/M^b)^{s+r-m}},
\label{2.22}\\
\llbracket K_{a,b}^{m,r}(X)\rrbracket&=\frac{K_{a,b}^{m,r}(X)}{(X/M^a)^m(X/M^b)^{s+r-m}}.
\label{2.23}
\end{align}
Using this notation, our earlier bounds for $J_{s+r}(X)$ may be rewritten in the form
\begin{equation}\label{2.24}
\llbracket J_{s+r}(X)\rrbracket >X^{\Lam-\del}\quad \text{and}\quad \llbracket 
J_{s+r}(Y)\rrbracket<Y^{\Lam+\del}\quad (Y\ge X^{1/2}),
\end{equation}
where $\Lam$ is defined by $\Lam=\lam-(s+r)$.\par

Finally, we recall two simple estimates associated with the system (\ref{1.1}).

\begin{lemma}\label{lemma2.1}
Suppose that $c$ is a non-negative integer with $c\tet\le 1$. Then for each natural number $u$, one has
$$\max_{1\le \xi\le p^c}\oint |\grf_c(\bfalp;\xi)|^{2u}\d\bfalp \ll_u J_u(X/M^c).$$
\end{lemma}

\begin{proof} This is \cite[Lemma 3.1]{Woo2012}.
\end{proof}

\begin{lemma}\label{lemma2.2}
Suppose that $c$ and $d$ are non-negative integers with $c\le \tet^{-1}$ and $d\le \tet^{-1}$. Then 
whenever $u,v\in \dbN$ and $\xi,\zet\in \dbZ$, one has
$$\oint |\grf_c(\bfalp;\xi)^{2u}\grf_d(\bfalp;\zet)^{2v}|\d\bfalp \ll_{u,v}(J_{u+v}(X/M^c))^{u/(u+v)}
(J_{u+v}(X/M^d))^{v/(u+v)}.$$
\end{lemma}

\begin{proof} This is \cite[Corollary 2.2]{FW2013}.
\end{proof}

\section{Auxiliary systems of congruences} Our goal of establishing strongly diagonal behaviour dictates 
that we must establish essentially diagonal behaviour for the solutions of auxiliary congruences central to 
our methods. Complicating our discussion in present circumstances is the potential for the ratio $b/a$ to 
be small, for in earlier work \cite{FW2013, Woo2013} this was assumed to be at least $\frac{1}{2}(k-1)$ 
or thereabouts. We begin by adjusting our notation to accommodate the demands of this paper. When $a$ 
and $b$ are integers with $1\le a<b$, we denote by $\calB_{a,b}^n(\bfm;\xi,\eta)$ the set of solutions of 
the system of congruences
\begin{equation}\label{3.1}
\sum_{i=1}^n(z_i-\eta)^j\equiv m_j\mmod{p^{jb}}\quad (1\le j\le k),
\end{equation}
with $1\le \bfz\le p^{kb}$ and $\bfz\equiv \bfxi\pmod{p^{a+1}}$ for some $\bfxi\in \Xi_a^n(\xi)$. We 
define an equivalence relation $\calR(\lam)$ on integral $n$-tuples by declaring $\bfx$ and $\bfy$ to be 
$\calR(\lam)$-equivalent when $\bfx\equiv \bfy\pmod{p^\lam}$. We then write 
$\calC_{a,b}^{n,h}(\bfm;\xi,\eta)$ for the set of $\calR(hb)$-equivalence classes of 
$\calB_{a,b}^n(\bfm;\xi,\eta)$, and we define $B_{a,b}^{n,h}(p)$ by putting
\begin{equation}\label{3.2}
B_{a,b}^{n,h}(p)=\max_{1\le \xi\le p^a}
\max_{\substack{1\le \eta\le p^b\\ \eta\not\equiv \xi\mmod{p}}}
\max_{1\le \bfm\le p^{kb}}\text{card}(\calC_{a,b}^{n,h}(\bfm;\xi,\eta)).
\end{equation}

\par When $a=0$ we modify these definitions, so that $\calB_{0,b}^n(\bfm;\xi,\eta)$ denotes the set of
 solutions of the system of congruences (\ref{3.1}) with $1\le \bfz\le p^{kb}$ and 
$\bfz\equiv \bfxi\mmod{p}$ for some $\bfxi\in \Xi_0^n(\xi)$, and for which in addition 
$\bfz\not\equiv \eta\mmod{p}$. As in the situation in which one has $a\ge 1$, we write 
$\calC_{0,b}^{n,h}(\bfm;\xi,\eta)$ for the set of $\calR(hb)$-equivalence classes of 
$\calB_{0,b}^n(\bfm;\xi,\eta)$, but we define $B_{0,b}^{n,h}(p)$ by putting
\begin{equation}\label{3.3}
B_{0,b}^{n,h}(p)=\max_{1\le \eta\le p^b}\max_{1\le \bfm\le p^{kb}}
\text{card}(\calC_{0,b}^{n,h}(\bfm;0,\eta)).
\end{equation}
We observe that the definition of $B_{a,b}^{n,h}(p)$ ensures that whenever $0<h'\le h$, one has
\begin{equation}\label{3.4}
B_{a,b}^{n,h'}(p)\le B_{a,b}^{n,h}(p).
\end{equation}

\begin{lemma}\label{lemma3.1}
Let $m$ be an integer with $0\le m\le k-1$, and suppose that $a$ and $b$ are integers with $0\le a<b$ 
and $(k-m)b\ge (m+1)a$. Suppose in addition that $h$ is a natural number with $a\le h\le (k-m)b-ma$. 
Then one has
$$B_{a,b}^{m+1,h/b}(p)\le k!.$$
\end{lemma}

\begin{proof} We are able to follow the arguments of the proofs of \cite[Lemma 3.5]{Woo2013} and 
\cite[Lemma 3.3]{FW2013}, with some modifications. For the sake of transparency of exposition, we 
provide an essentially complete account. We begin by considering the situation with $a\ge 1$, the 
alterations required when $a=0$ being easily accommodated within our basic argument. Consider 
fixed natural numbers $a$ and $b$ with $1\le a<b$ and $(k-m)b\ge (m+1)a$, and fixed integers $\xi$ 
and $\eta$ with $1\le \xi\le p^a$, $1\le \eta\le p^b$ and $\eta\not\equiv \xi\mmod{p}$. Note that in 
view of the relation (\ref{3.4}), the conclusion of the lemma follows in general if we can establish it when 
$h=(k-m)b-ma$. We henceforth assume that the latter is the case. We then denote by $\calD_1(\bfn)$ the 
set of $\calR(h)$-equivalence classes of solutions of the system of congruences
\begin{equation}\label{3.5}
\sum_{i=1}^{m+1}(z_i-\eta)^j\equiv n_j\mmod{p^{(k-m)b}}\quad (k-m\le j\le k),
\end{equation}
with $1\le \bfz\le p^{kb}$ and $\bfz\equiv \bfxi\mmod{p^{a+1}}$ for some 
$\bfxi\in \Xi_a^{m+1}(\xi)$. To any solution $\bfz$ of (\ref{3.1}) there corresponds a unique 
$(m+1)$-tuple $(n_{k-m},\ldots ,n_k)$, with $1\le \bfn\le p^{(k-m)b}$, for which (\ref{3.5}) holds and
$$n_j\equiv m_j\mmod{p^{(k-m)b}}\quad (k-m\le m\le k).$$
Since $h/b\le k-m$, we therefore infer that
\begin{equation}\label{3.6}
\text{card}(\calC_{a,b}^{m+1,h/b}(\bfm;\xi,\eta))\le \max_{1\le \bfn\le p^{(k-m)b}}
\text{card}(\calD_1(\bfn)).
\end{equation} 

\par We fix any choice of $\bfn$ for which the maximum is achieved in (\ref{3.6}). There is plainly no 
loss of generality in supposing that $\calD_1(\bfn)$ is non-empty. Observe that for any solution $\bfz'$ of 
(\ref{3.5}) there is an $\calR(k-m)$-equivalent solution $\bfz$ satisfying $1\le \bfz\le p^{(k-m)b}$. 
Rewriting each variable $z_i$ in (\ref{3.5}) in the shape $z_i=p^ay_i+\xi$, we infer from the hypothesis 
that $\bfz\equiv \bfxi\mmod{p^{a+1}}$ for some $\bfxi\in \Xi_a^{m+1}(\xi)$ that the $(m+1)$-tuple 
$\bfy$ necessarily satisfies
\begin{equation}\label{3.7}
y_i\not\equiv y_l\mmod{p}\quad (1\le i<l\le m+1).
\end{equation}
Write $\zet=\xi-\eta$, and note that the constraint $\eta\not\equiv \xi\mmod{p}$ ensures that 
$p\nmid \zet$. We denote the multiplicative inverse of $\zet$ modulo $p^{(k-m)b}$ by $\zet^{-1}$. Then 
we deduce from (\ref{3.5}) that $\text{card}(\calD_1(\bfn))$ is bounded above by the number of 
$\calR(h-a)$-equivalence classes of solutions of the system of congruences
\begin{equation}\label{3.8}
\sum_{i=1}^{m+1}(p^ay_i\zet^{-1}+1)^j\equiv n_j(\zet^{-1})^j\mmod{p^{(k-m)b}}\quad 
(k-m\le j\le k),
\end{equation}
with $1\le \bfy\le p^{(k-m)b-a}$ satisfying (\ref{3.7}). Let $\bfy=\bfw$ be any solution of the system 
(\ref{3.8}). Then we find that all other solutions $\bfy$ satisfy the system
\begin{equation}\label{3.9}
\sum_{i=1}^{m+1} \left( (p^ay_i\zet^{-1}+1)^j-(p^aw_i\zet^{-1}+1)^j\right) \equiv 0
\mmod{p^{(k-m)b}}\quad (k-m\le j\le k).
\end{equation}

We next apply \cite[Lemma 3.2]{Woo2013}, just as in the argument of the proof of 
\cite[Lemmata 3.3 to 3.6]{Woo2013}. Consider an index $j$ with $k-m\le j\le k$, and apply the latter 
lemma with $\alp=k-m-1$ and $\bet=j-k+m+1$. We find that there exist integers $c_{j,l}$ 
$(k-m-1\le l\le j)$ and $d_{j,u}$ $(j-k+m+1\le u\le j)$, bounded in terms of $k$, and with 
$d_{j,j-k+m+1}\ne 0$, for which one has the polynomial identity
\begin{equation}\label{3.10}
c_{j,k-m-1}+\sum_{l=k-m}^jc_{j,l}(x+1)^l=\sum_{u=j-k+m+1}^jd_{j,u}x^u.
\end{equation}
Since we may assume $p$ to be large, we may suppose that $p\nmid d_{j,j-k+m+1}$. Thus, multiplying 
(\ref{3.10}) through by the multiplicative inverse of $d_{j,j-k+m+1}$ modulo $p^{(k-m)b}$, we see that 
there is no loss in supposing that
$$d_{j,j-k+m+1}\equiv 1\mmod{p^{(k-m)b}}.$$
Consequently, by taking suitable linear combinations of the congruences comprising (\ref{3.9}), we discern 
that any solution of this system satisfies
$$(\zet^{-1}p^a)^{j-k+m+1}\sum_{i=1}^{m+1}(\psi_j(y_i)-\psi_j(w_i))\equiv 0
\mmod{p^{(k-m)b}}\quad (k-m\le j\le k),$$
in which
$$\psi_j(z)=z^{j-k+m+1}+\sum_{u=j-k+m+2}^jd_{j,u}(\zet^{-1}p^a)^{u-j+k-m-1}z^u.$$
Note here that
\begin{equation}\label{3.11}
\psi_j(z)\equiv z^{j-k+m+1}\mmod{p}.
\end{equation}

\par Denote by $\calD_2(\bfu)$ the set of $\calR(h-a)$-equivalence classes of solutions of the system of 
congruences
$$\sum_{i=1}^{m+1}\psi_j(y_i)\equiv u_j\mmod{p^{(k-m)b-(j-k+m+1)a}}\quad (k-m\le j\le k),$$
with $1\le \bfy\le p^{(k-m)b-a}$ satisfying (\ref{3.7}). Then we have shown thus far that
\begin{equation}\label{3.12}
\text{card}(\calD_1(\bfn))\le \max_{1\le \bfu\le p^{(k-m)b}}\text{card}(\calD_2(\bfu)).
\end{equation}
Let $\calD_3(\bfv)$ denote the set of solutions of the congruence
$$\sum_{i=1}^{m+1}\psi_j(y_i)\equiv v_j\mmod{p^{h-a}}\quad (k-m\le j\le k),$$
with $1\le \bfy\le p^{h-a}$ satisfying (\ref{3.7}). For $k-m\le j\le k$, we have
$$(k-m)b-(j-k+m+1)a\ge (k-m)b-(m+1)a=h-a.$$
Then we arrive at the upper bound
\begin{equation}\label{3.13}
\text{card}(\calD_2(\bfu))\le \max_{1\le \bfv\le p^{h-a}}\text{card}(\calD_3(\bfv)).
\end{equation}

\par By combining (\ref{3.6}), (\ref{3.12}) and (\ref{3.13}), we discern at this point that
\begin{equation}\label{3.14}
\text{card}(\calC_{a,b}^{m+1,h/b}(\bfm;\xi,\eta))\le \max_{1\le \bfv\le p^{h-a}}
\text{card}(\calD_3(\bfv)).
\end{equation}
Define the determinant
$$J(\bfpsi;\bfx)=\det \left(\psi'_{k-m+l-1}(x_i)\right)_{1\le i,l\le m+1}.$$
In view of (\ref{3.11}), one has $\psi'_{k-m+l-1}(y_i)\equiv ly_i^{l-1}\mmod{p}$. It follows 
from (\ref{3.7}) that
$$\det (y_i^{l-1})_{1\le i,l\le m+1}=\prod_{1\le i<u\le m+1}(y_i-y_u)\not\equiv 0\mmod{p},$$
so that, since $p>k$, we have $(J(\bfpsi;\bfy),p)=1$. We therefore deduce from 
\cite[Theorem 1]{Woo1996a}, just as in the corresponding argument of the proof of 
\cite[Lemma 3.3]{Woo2013} following \cite[equation (3.17)]{Woo2013}, that
$$\text{card}(\calD_3(\bfv))\le (k-m)(k-m+1)\ldots k\le k!,$$
and thus the conclusion of the lemma when $a\ge 1$ follows at once from (\ref{3.2}) and (\ref{3.14}).

\par When $a=0$ we must apply some minor modifications to the above argument. In this case, we 
denote by $\calD_1(\bfn;\eta)$ the set of $\calR(h)$-equivalence classes of solutions of the system of 
congruences (\ref{3.5}) with $1\le \bfz\le p^{kb}$ and $\bfz\equiv \bfxi\mmod{p}$ for some 
$\bfxi\in \Xi_0^{m+1}(0)$, and for which in addition $z_i\not \equiv \eta\mmod{p}$ for 
$1\le i\le m+1$. Then as in the opening paragraph of our proof, it follows from (\ref{3.1}) that
\begin{equation}\label{3.15}
\text{card}(\calC_{0,b}^{m+1,h/b}(\bfm;0,\eta))\le \max_{1\le \bfn\le p^{(k-m)b}}
\text{card}(\calD_1(\bfn;\eta)).
\end{equation}
But $\text{card}(\calD_1(\bfn;\eta))=\text{card}(\calD_1(\bfn;0))$, and $\text{card}(\calD_1(\bfn;0))$ 
counts the solutions of the system of congruences
$$\sum_{i=1}^{m+1}y_i^j\equiv n_j\mmod{p^{(k-m)b}}\quad (k-m\le j\le k),$$
with $1\le \bfy\le p^{(k-m)b}$ satisfying (\ref{3.7}), and in addition $p\nmid y_i$ $(1\le i\le m+1)$. 
Write
$$J(\bfy)=\det \left( (k-m+j-1)y_i^{k-m+j-2}\right)_{1\le i,j\le m+1}.$$
Then, since $p>k$, we have
$$J(\bfy)=\frac{k!}{(k-m-1)!}(y_1\ldots y_{m+1})^{k-m-1}\prod_{1\le i<j\le m+1}(y_i-y_j)
\not \equiv 0\mmod{p}.$$
We therefore conclude from \cite[Theorem 1]{Woo1996a} that
$$\text{card}(\calD_1(\bfn;0))\le (k-m)(k-m+1)\ldots k\le k!.$$
In view of (\ref{3.3}), the conclusion of the lemma therefore follows from (\ref{3.15}) when $a=0$. 
This completes our account of the proof of the lemma.
\end{proof}
 
\section{The conditioning process} The mean value $K_{a,b}^{m,r}(X;\xi,\eta)$ differs only slightly from 
the special case $K_{a,b}^{k-1,k-1}(X;\xi,\eta)$ considered in \cite{Woo2014}, and thus our treatment of 
the conditioning process may be swiftly executed.

\begin{lemma}\label{lemma4.1} Let $a$ and $b$ be integers with $b>a\ge 1$. Then whenever $s\ge 2r-1$, 
one has
$$I_{a,b}^{r,r}(X)\ll K_{a,b}^{r,r}(X)+M^{2s/3}I_{a,b+1}^{r,r}(X).$$
\end{lemma}

\begin{proof} Our argument follows the proof of \cite[Lemma 4.1]{Woo2014} with minor adjustments. 
Consider fixed integers $\xi$ and $\eta$ with $\eta\not\equiv \xi\mmod{p}$. Let $T_1$ denote the 
number of integral solutions $\bfx$, $\bfy$, $\bfv$, $\bfw$ of the system (\ref{2.13}) counted by 
$I_{a,b}^{r,r}(X;\xi,\eta)$ in which $v_1,\ldots,v_s$ together occupy at least $r$ distinct residue 
classes modulo $p^{b+1}$, and let $T_2$ denote the corresponding number of solutions in which 
these integers together occupy at most $r-1$ distinct residue classes modulo $p^{b+1}$. Then
\begin{equation}\label{4.1}
I_{a,b}^{r,r}(X;\xi,\eta)=T_1+T_2.
\end{equation}

\par From (\ref{2.10}) to (\ref{2.12}), orthogonality and Schwarz's inequality, one has
\begin{align}
T_1&\le \binom{s}{r}\oint |\grF_a^r(\bfalp;\xi)|^2\grF_b^r(\bfalp;\eta)
\grf_b(\bfalp;\eta)^{s-r}\grf_b(-\bfalp;\eta)^s\d\bfalp \notag \\
&\ll \left( K_{a,b}^{r,r}(X;\xi,\eta)\right)^{1/2}\left( I_{a,b}^{r,r}(X;\xi,\eta)\right)^{1/2}.\label{4.2}
\end{align}
In order to treat $T_2$, we observe first that the hypothesis $s\ge 2r-1$ ensures that there is an integer 
$\zet\equiv \eta \mmod{p^b}$ having the property that three at least of the variables $v_1,\ldots ,v_s$ 
are congruent to $\zet$ modulo $p^{b+1}$. Hence, recalling the definitions (\ref{2.10}) and 
(\ref{2.11}), one finds from orthogonality and H\"older's inequality that
\begin{align}
T_2&\le \binom{s}{3}\sum_{\substack{1\le \zet\le p^{b+1}\\ \zet\equiv \eta\mmod{p^b}}}
\oint |\grF_a^r(\bfalp;\xi)|^2\grf_{b+1}(\bfalp;\zet)^3\grf_b(\bfalp;\eta)^{s-3}\grf_b(-\bfalp;\eta)^s
\d\bfalp \notag \\
&\ll M\max_{\substack{1\le \zet \le p^{b+1}\\ \zet\equiv \eta \mmod{p^b}}}
(I_{a,b}^{r,r}(X;\xi,\eta))^{1-3/(2s)}(I_{a,b+1}^{r,r}(X;\xi,\zet))^{3/(2s)}.\label{4.3}
\end{align}

\par By substituting (\ref{4.2}) and (\ref{4.3}) into (\ref{4.1}), and recalling (\ref{2.16}) and
 (\ref{2.17}), we therefore conclude that
$$I_{a,b}^{r,r}(X)\ll (K_{a,b}^{r,r}(X))^{1/2}(I_{a,b}^{r,r}(X))^{1/2}+
M(I_{a,b}^{r,r}(X))^{1-3/(2s)}(I_{a,b+1}^{r,r}(X))^{3/(2s)},$$
whence
$$I_{a,b}^{r,r}(X)\ll K_{a,b}^{r,r}(X)+M^{2s/3}I_{a,b+1}^{r,r}(X).$$
This completes the proof of the lemma.
\end{proof}

Repeated application of Lemma \ref{lemma4.1} combines with a trivial bound for the mean value 
$K_{a,b+H}^{r,r}(X)$ to deliver the basic conditioning lemma of this section.

\begin{lemma}\label{lemma4.2} Let $a$ and $b$ be integers with $1\le a<b$, and let $H$ be any integer 
with $H\ge 15$. Suppose that $b+H\le (2\tet)^{-1}$. Then there exists an integer $h$ with $0\le h<H$ 
having the property that
$$I_{a,b}^{r,r}(X)\ll (M^h)^{2s/3}K_{a,b+h}^{r,r}(X)+(M^H)^{-s/4}(X/M^b)^s(X/M^a)^{\lam-s}.$$
\end{lemma}

\begin{proof} Repeated application of Lemma \ref{lemma4.1} yields the upper bound
\begin{equation}\label{4.4}
I_{a,b}^{r,r}(X)\ll \sum_{h=0}^{H-1}(M^h)^{2s/3}K_{a,b+h}^{r,r}(X)+
(M^H)^{2s/3}I_{a,b+H}^{r,r}(X).
\end{equation}
On considering the underlying Diophantine systems, it follows from Lemma \ref{lemma2.2} that for each 
$\xi$ and $\eta$, one has
\begin{align*}
I_{a,b+H}^{r,r}(X;\xi,\eta)&\le \oint|\grf_a(\bfalp;\xi)^{2r}\grf_{b+H}(\bfalp;\eta)^{2s}|\d\bfalp \\
&\ll \left(J_{s+r}(X/M^a)\right)^{r/(s+r)}\left(J_{s+r}(X/M^{b+H})\right)^{s/(s+r)}.
\end{align*}
Since $M^{b+H}=(X^\tet)^{b+H}\le X^{1/2}$, we deduce from (\ref{2.16}) and (\ref{2.24}) that
\begin{align*}
(M^H)^{2s/3}I_{a,b+H}^{r,r}(X)&\ll X^\del \left( (X/M^a)^{r/(s+r)}(X/M^{b+H})^{s/(s+r)}
\right)^\lam (M^H)^{2s/3}\\
&=X^\del (X/M^b)^s(X/M^a)^{\lam-s}M^\Ome ,
\end{align*}
where
$$\Ome=\lam \left( a-\frac{ar}{s+r}-\frac{bs}{s+r}\right) +s(b-a)+Hs\left( \frac{2}{3}-\frac{\lam}{s+r}\right) .$$
Since $\lam\ge s+r$, the lower bound $b>a$ leads to the estimate
$$\Ome\le -s(b-a)\frac{\lam}{s+r}+s(b-a)-\tfrac{1}{3}Hs\le -\tfrac{1}{3}Hs.$$
This completes the proof of the lemma, since with $H\ge 15$ one has
$$X^\del M^{-Hs/3}\ll M^{-Hs/4}.$$
\end{proof}

\section{The precongruencing step} The alteration of the definition of $K_{a,b}^{m,r}(X)$ relative to our 
earlier treatments does not prevent previous precongruencing arguments from applying, mutatis mutandis. 
However, we seek to ensure that our initial value of $b$ is very large, so that subsequent iterations work 
efficiently. For this reason, the argument of the proof of \cite[Lemma 6.1]{FW2013} must be modified in 
several ways that demand a fairly complete account of the proof. We recall the definition (\ref{2.8}) of $B$.

\begin{lemma}\label{lemma5.1}
There exists a prime number $p$, with $M<p\le 2M$, and an integer $h$ with $0\le h\le 4B$, for which 
one has
$$J_{s+r}(X)\ll M^{2sB+2sh/3}K_{0,B+h}^{r,r}(X).$$
\end{lemma}

\begin{proof} The mean value $J_{s+r}(X)$ counts the number of integral solutions of the system
\begin{equation}\label{5.1}
\sum_{i=1}^{s+r}(x_i^j-y_i^j)=0\quad (1\le j\le k),
\end{equation}
with $1\le \bfx,\bfy\le X$. Let $\calP$ denote a set of $\lceil (s+r)^2\tet^{-1}\rceil$ prime numbers 
in $(M,2M]$. That such a set exists is a consequence of the Prime Number Theorem. The argument of 
the proof of \cite[Lemma 6.1]{FW2013} leading to equation (\ref{6.2}) of that paper shows that for 
some $p\in \calP$, one has $J_{s+r}(X)\ll T(p)$, where $T(p)$ denotes the number of solutions of the 
system (\ref{5.1}) counted by $J_{s+r}(X)$ in which $x_1,\ldots ,x_{s+r}$ are distinct modulo $p$, 
and likewise $y_1,\ldots ,y_{s+r}$ are distinct modulo $p$.\par

We now examine the residue classes modulo $p^B$ of a given solution $\bfx,\bfy$ counted by $T(p)$. 
Let $\bfeta$ and $\bfzet$ be $s$-tuples with $1\le \bfeta,\bfzet\le p^B$ satisfying the condition that for 
$1\le i\le s$, one has $x_i\equiv \eta_i\mmod{p^B}$ and $y_i\equiv \zet_i\mmod{p^B}$. Recall the 
notation introduced prior to (\ref{2.18}). Then since $x_1,\ldots ,x_{s+r}$ are distinct modulo $p$, it 
follows that $(x_{s+1},\ldots ,x_{s+r})\in \Xi^r(\bfeta)$, and likewise one finds that 
$(y_{s+1},\ldots ,y_{s+r})\in \Xi^r(\bfzet)$. Then on considering the underlying Diophantine systems, 
one finds that
$$T(p)\le \sum_{1\le \bfeta,\bfzet\le p^B}\oint \left( \prod_{i=1}^s\grf_B(\bfalp;\eta_i)
\grf_B(-\bfalp;\zet_i)\right) \grF^r(\bfalp;\bfeta)\grF^r(-\bfalp;\bfzet)\d\bfalp .$$
Write  
$$\calJ(\bftet,\psi)=\oint |\grF^r(\bfalp;\bftet)^2\grf_B(\bfalp;\psi)^{2s}|\d\bfalp .$$
Then by applying H\"older's inequality, and again considering the underlying Diophantine systems, we 
discern that
\begin{align*}
T(p)&\le \sum_{1\le \bfeta,\bfzet \le p^B}\prod_{i=1}^s \left( \calJ(\bfeta,\eta_i)\calJ(\bfzet,\zet_i)
\right)^{1/(2s)}\\
&\le \sum_{1\le \bfeta,\bfzet\le p^B}\prod_{i=1}^s\left( \calJ(\eta_i,\eta_i)\calJ(\zet_i,\zet_i)\right)^{1/(2s)}.
\end{align*}
Hence, on recalling the definition (\ref{2.18}), we obtain the upper bound
\begin{align}
T(p)&\le p^{2sB}\max_{1\le \eta\le p^B}\oint |\grF^r(\bfalp;\eta)^2\grf_B(\bfalp;\eta)^{2s}|\d\bfalp \notag \\
&=p^{2sB}\max_{1\le \eta\le p^B}\Itil_B^{r,r}(X;\eta).\label{5.2}
\end{align}
By modifying the argument of the proof of \cite[Lemma 6.1]{FW2013} leading from equation (\ref{6.3}) to equation 
(\ref{6.6}) of that paper, along the lines easily surmised from our proof of Lemma \ref{lemma4.1} above, one finds that
\begin{equation}\label{5.3}
\Itil_c^{r,r}(X;\eta)\ll \Ktil_c^{r,r}(X;\eta)+M^{2s/3}\max_{1\le \zet\le p^{c+1}}\Itil_{c+1}^{r,r}(X;\zet).
\end{equation}

\par Iterating (\ref{5.3}) in order to bound $\Itil_B^{r,r}(X;\eta)$, just as in the argument concluding the 
proof of \cite[Lemma 6.1]{FW2013} (and see also (\ref{4.4}) above), we discern from (\ref{5.2}) that
\begin{align}
J_{s+r}(X)\ll &\,M^{2sB}\max_{0\le h\le 4B}(M^h)^{2s/3}\Ktil_{B+h}^{r,r}(X)\notag \\
&\,+M^{2sB+8sB/3}\max_{1\le \zet\le p^{5B+1}}\Itil_{5B+1}^{r,r}(X;\zet).\label{5.4}
\end{align}
By considering the underlying Diophantine systems, we deduce from Lemma \ref{lemma2.2} that
$$\Itil_{5B+1}^{r,r}(X;\zet)\ll (J_{s+r}(X))^{r/(s+r)}(J_{s+r}(X/M^{5B+1}))^{s/(s+r)}.$$
In this way, we find from (\ref{5.4}) either that
\begin{equation}\label{5.5}
J_{s+r}(X)\ll M^{2sB+2sh/3}\Ktil_{B+h}^{r,r}(X)
\end{equation}
for some index $h$ with $0\le h\le 4B$, so that the conclusion of the lemma holds, or else that
$$J_{s+r}(X)\ll M^{14sB/3}(J_{s+r}(X))^{r/(s+r)}(J_{s+r}(X/M^{5B+1}))^{s/(s+r)}.$$
In the latter case, since $\lam\ge s+r$, we obtain the upper bound
\begin{align*}
J_{s+r}(X)&\ll M^{14(s+r)B/3}J_{s+r}(X/M^{5B+1})\\
&\ll M^{14(s+r)B/3}(X/M^{5B+1})^{\lam+\del}\\
&\ll X^{\lam+\del}M^{-(s+r)B/3}.
\end{align*}
Invoking the definition (\ref{2.8}) of $\del$, we find that $J_{s+r}(X)\ll X^{\lam-2\del}$, contradicting 
the lower bound (\ref{2.24}) if $X=X_l$ is large enough. We are therefore forced to accept the former 
upper bound (\ref{5.5}), and hence the proof of the lemma is completed by reference to (\ref{2.20}).
\end{proof}

At this point, we fix the prime number $p$, once and for all, in accordance with Lemma \ref{lemma5.1}.

\section{The efficient congruencing step} Our strategy for executing the efficient congruencing process is 
based on that in our recent work \cite[\S6]{Woo2014}, though in present circumstances only one phase 
is required, relating $K_{a,b}^{m+1,r}(X)$ to $K_{a,b}^{m,r}(X)$ and $I_{b,b'}^{r,r}(X)$, for a 
suitable value of $b'$. We begin our discussion of the efficient congruencing process with some additional 
notation. We define the generating function
\begin{equation}\label{6.1}
\grH_{c,d}^m(\bfalp;\xi)=\sum_{\bfxi\in \Xi_c^m(\xi)}
\sum_{\substack{1\le \bfzet\le p^d\\ \bfzet\equiv \bfxi\mmod{p^{c+1}}}}\prod_{i=1}^m
|\grf_d(\bfalp;\zet_i)|^2,
\end{equation}
in which we adopt the natural convention that $\grH_{c,d}^0(\bfalp;\xi)=1$. It is useful for future 
reference to record the consequence of H\"older's inequality given in \cite[equation (6.2)]{Woo2014}, 
namely that whenever $\ome$ is a real number with $m\ome\ge 1$, then
\begin{equation}\label{6.2}
\grH_{c,d}^m(\bfalp;\xi)^\ome \le (p^{d-c})^{m\ome -1}
\sum_{\substack{1\le \zet \le p^d\\ \zet \equiv \xi\mmod{p^c}}}|\grf_d(\bfalp;\zet)|^{2m\ome}.
\end{equation}

\begin{lemma}\label{lemma6.1} Let $m$ be an integer with $0\le m\le r-1$. Suppose that $a$ and $b$ 
are integers with $0\le a<b\le \tet^{-1}$ and $(k-m)b\ge (m+1)a$, and suppose further that 
$a\le b/\sqrt{k}$. Then whenever $b'$ is an integer with $a\le b'\le (k-m)b-ma$, one has
$$K_{a,b}^{m+1,r}(X)\ll \left( (M^{b'-a})^sI_{b,b'}^{r,r}(X)\right)^{1/(s-m)}
\left( K_{a,b}^{m,r}(X)\right)^{(s-m-1)/(s-m)}.$$
\end{lemma}

\begin{proof} We first consider the situation in which $a\ge 1$. The argument associated with the case 
$a=0$ is very similar, and so we are able to appeal later to a highly abbreviated argument for this case in 
order to complete the proof of the lemma. Consider fixed integers $\xi$ and $\eta$ with
\begin{equation}\label{6.3}
1\le \xi\le p^a,\quad 1\le \eta\le p^b\quad \text{and}\quad \eta\not\equiv \xi\mmod{p}.
\end{equation}
The quantity $K_{a,b}^{m+1,r}(X;\xi,\eta)$ counts the number of integral solutions of the system 
(\ref{2.14}) with $m$ replaced by $m+1$, subject to the attendant conditions on $\bfx$, $\bfy$, $\bfu$, 
$\bfv$, $\bfw$, $\bfz$. Given such a solution of the system (\ref{2.14}), the discussion leading to 
(\ref{2.15}) shows that
\begin{equation}\label{6.4}
\sum_{i=1}^{m+1}(x_i-\eta)^j\equiv \sum_{i=1}^{m+1}(y_i-\eta)^j\mmod{p^{jb}}\quad 
(1\le j\le k).
\end{equation}
In the notation introduced in \S3, it follows that for some $k$-tuple of integers $\bfm$, both 
$[\bfx\mmod{p^{kb}}]$ and $[\bfy\mmod{p^{kb}}]$ lie in $\calB_{a,b}^{m+1}(\bfm;\xi,\eta)$. Write
$$\grG_{a,b}(\bfalp;\bfm)=\sum_{\bftet\in \calB_{a,b}^{m+1}(\bfm;\xi,\eta)} 
\prod_{i=1}^{m+1}\grf_{kb}(\bfalp;\tet_i).$$
Then on considering the underlying Diophantine system, we see from (\ref{2.12}) and (\ref{6.4}) that 
$$K_{a,b}^{m+1,r}(X;\xi,\eta)=\sum_{m_1=1}^{p^b}\ldots \sum_{m_k=1}^{p^{kb}}
\oint |\grG_{a,b}(\bfalp;\bfm)^2\grF_m^*(\bfalp)^2|\d\bfalp ,$$
where we write
\begin{equation}\label{6.5}
\grF_m^*(\bfalp)=\grF_b^r(\bfalp;\eta)\grf_b(\bfalp;\eta)^{s-m-1}.
\end{equation}

\par We now partition the vectors in each set $\calB_{a,b}^{m+1}(\bfm;\xi,\eta)$ into equivalence 
classes modulo $p^{b'}$ as in \S3. Write $\calC(\bfm)=\calC_{a,b}^{m+1,b'/b}(\bfm;\xi,\eta)$. An 
application of Cauchy's inequality leads via (\ref{3.2}) and Lemma \ref{lemma3.1} to the bound
\begin{align*}
|\grG_{a,b}(\bfalp;\bfm)|^2&=\Bigl| \sum_{\grC \in \calC(\bfm)}
\sum_{\bftet \in \grC}\prod_{i=1}^{m+1}\grf_{kb}(\bfalp;\tet_i)\Bigr|^2\\
&\le \text{card}(\calC(\bfm))\sum_{\grC\in \calC(\bfm)}
\Bigl| \sum_{\bftet \in \grC}\prod_{i=1}^{m+1}\grf_{kb}(\bfalp;\tet_i)\Bigr|^2\\
&\le k!\sum_{\grC\in \calC(\bfm)}\Bigl| \sum_{\bftet \in \grC}\prod_{i=1}^{m+1}
\grf_{kb}(\bfalp;\tet_i)\Bigr|^2.
\end{align*}
It may be helpful to note here that since $b'/b\le k-m-ma/b$, then in view of the relation 
(\ref{3.2}), it follows from Lemma \ref{lemma3.1} that
$$\text{card}(\calC(\bfm))=B_{a,b}^{m+1,b'/b}(p)\le k!.$$
Hence
$$K_{a,b}^{m+1,r}(X;\xi,\eta)\ll \sum_\bfm \sum_{\grC\in \calC(\bfm)}\oint 
\Bigl| \grF_m^*(\alp)\sum_{\bftet \in \grC}\prod_{i=1}^{m+1}\grf_{kb}(\bfalp;\tet_i)\Bigr|^2
\d \bfalp.$$
For each $k$-tuple $\bfm$ and equivalence class $\grC$, the integral above counts solutions of 
(\ref{2.14}) with the additional constraint that both $[\bfx \mmod{p^{kb}}]$ and 
$[\bfy \mmod{p^{kb}}]$ lie in $\grC$. In particular, one has $\bfx\equiv \bfy\mmod{p^{b'}}$. 
Moreover, as the sets $\calB_{a,b}^{m+1}(\bfm;\xi,\eta)$ are disjoint for distinct $k$-tuples $\bfm$ with 
$1\le m_j\le p^{jb}$ $(1\le j\le k)$, to each pair $(\bfx,\bfy)$ there corresponds at most one pair 
$(\bfm,\grC)$. Thus we deduce that 
$$K_{a,b}^{m+1,r}(X;\xi,\eta)\ll H^\dagger (X;\xi,\eta),$$
where $H^\dagger(X;\xi,\eta)$ denotes the number of solutions of (\ref{2.14}) subject to the additional 
condition $\bfx\equiv \bfy\mmod{p^{b'}}$. Hence, on considering the underlying Diophantine systems and 
recalling (\ref{6.1}), we discern that
\begin{equation}\label{6.6}
K_{a,b}^{m+1,r}(X;\xi,\eta)\ll \oint \grH_{a,b'}^{m+1}(\bfalp;\xi)|\grF_m^*(\bfalp)|^2\d\bfalp .
\end{equation}

\par An inspection of the definition of $\Xi_a^m(\xi)$ in the preamble to (\ref{2.10}) reveals on this 
occasion that when $\bfxi\in \Xi_a^{m+1}(\xi)$, then
$$(\xi_1,\ldots ,\xi_m)\in \Xi_a^m(\xi)\quad \text{and}\quad (\xi_{m+1})\in \Xi_a^1(\xi).$$
Then a further consideration of the underlying Diophantine systems leads from (\ref{6.6}) via (\ref{6.1}) 
to the upper bound
$$K_{a,b}^{m+1,r}(X;\xi,\eta)\ll \oint \grH^m_{a,b'}(\bfalp;\xi)\grH^1_{a,b'}(\bfalp;\xi)|
\grF_m^*(\bfalp)|^2\d\bfalp .$$
By applying H\"older's inequality to the integral on the right hand side of this relation, bearing in mind the 
definition (\ref{6.5}), we obtain the bound
\begin{equation}\label{6.7}
K_{a,b}^{m+1,r}(X;\xi,\eta)\ll U_1^{\ome_1}U_2^{\ome_2}U_3^{\ome_3},
\end{equation}
where
\begin{equation}\label{6.8}
\ome_1=\frac{s-m-1}{s-m},\quad \ome_2=\frac{1}{s},\quad \ome_3=\frac{m}{s(s-m)},
\end{equation}
and
\begin{equation}\label{6.9}
U_1=\oint \grH_{a,b'}^m(\bfalp;\xi)|\grF_b^r(\bfalp;\eta)^2\grf_b(\bfalp;\eta)^{2s-2m}|\d\bfalp ,
\end{equation}
\begin{equation}\label{6.10}
U_2=\oint |\grF_b^r(\bfalp;\eta)|^2\grH_{a,b'}^1(\bfalp;\xi)^s\d\bfalp ,
\end{equation}
\begin{equation}\label{6.11}
U_3=\oint |\grF_b^r(\bfalp;\eta)|^2\grH_{a,b'}^m(\bfalp;\xi)^{s/m}\d\bfalp .
\end{equation}

\par We now relate the mean values $U_i$ to those introduced in \S2. Observe first that a consideration of 
the underlying Diophantine system leads from (\ref{6.9}) via (\ref{6.1}) and (\ref{2.10}) to the upper 
bound
$$U_1\le \oint |\grF_a^m(\bfalp;\xi)^2\grF_b^r(\bfalp;\eta)^2\grf_b(\bfalp;\eta)^{2s-2m}|\d\bfalp.$$
On recalling (\ref{2.12}) and (\ref{2.17}), we thus deduce that
\begin{equation}\label{6.12}
U_1\le K_{a,b}^{m,r}(X).
\end{equation}
Next, by employing (\ref{6.2}) within (\ref{6.10}) and (\ref{6.11}), we find that
$$U_2+U_3\ll (M^{b'-a})^s\max_{\substack{1\le \zet\le p^{b'}\\ \zet\equiv \xi\mmod{p^a}}}
\oint |\grF_b^r(\bfalp ;\eta)^2\grf_{b'}(\bfalp;\zet)^{2s}|\d\bfalp .$$
Notice here that since the condition (\ref{6.3}) implies that $\eta\not\equiv \xi\mmod{p}$, and we have 
$\zet\equiv \xi\mmod{p^a}$ with $a\ge 1$, then once more one has $\zet\not\equiv \eta\mmod{p}$. In 
this way we deduce from (\ref{2.11}) and (\ref{2.16}) that
\begin{equation}\label{6.13}
U_2+U_3\ll (M^{b'-a})^s I_{b,b'}^{r,r}(X).
\end{equation}

By substituting (\ref{6.12}) and (\ref{6.13}) into the relation
$$K_{a,b}^{m+1,r}(X;\xi,\eta )\ll U_1^{\ome_1}(U_2+U_3)^{1-\ome_1},$$
that is immediate from (\ref{6.7}), and then recalling (\ref{6.8}) and (\ref{2.17}), the conclusion of the 
lemma follows when $a\ge 1$. When $a=0$, we must modify this argument slightly. In this case, from 
(\ref{2.19}) and (\ref{2.20}) we find that
$$K_{0,b}^{m+1,r}(X)=\max_{1\le \eta\le p^b}\oint |\grF^{m+1}(\bfalp;\eta)^2
\grF_b^r(\bfalp;\eta)^2\grf_b(\bfalp;\eta)^{2s-2m-2}|\d \bfalp.$$
The desired conclusion follows in this instance by pursuing the proof given above in the case $a\ge 1$,
 noting that the definition of $\grF^{m+1}(\bfalp;\eta)$ ensures that the variables resulting from the 
congruencing argument avoid the congruence class $\eta$ modulo $p$. This completes our account of 
the proof of the lemma.
\end{proof}

We note that in our application of Lemma \ref{lemma6.1}, we restrict to situations with $0\le m\le r-1$. 
Thus, since $r\le r_0$, it follows from (\ref{2.2}) that the hypothesis $a\le b/\sqrt{k}$ ensures that
$$k-m-ma/b\ge k-(k-2\sqrt{k}+1)-(k-2\sqrt{k}+1)/\sqrt{k}>\sqrt{k}.$$
Since $b$ is large, we are therefore at liberty to apply Lemma \ref{lemma6.1} with a choice for $b'$ 
satisfying the condition $b/b'<1/\sqrt{k}$, thereby preparing appropriately for subsequent applications 
of Lemma \ref{lemma6.1}.

\section{The multigrade combination} We next combine the estimates supplied by Lemma \ref{lemma6.1} 
so as to bound $K_{a,b}^{r,r}(X)$ in terms of the mean values $I_{b,k_mb}^{r,r}(X)$ $(0\le m\le r-1)$, 
in which the integers $k_m$ satisfy $k_m\le k-m-\lceil m/\sqrt{k}\rceil$. Before announcing our basic 
asymptotic estimate, we define the exponents
\begin{equation}\label{7.1}
s_m=s-m\quad \text{and}\quad \phi_m=(s-r)/(s_ms_{m+1})\quad (0\le m\le r-1).
\end{equation}
In addition, we write
\begin{equation}\label{7.2}
\phi^*=(s-r)/s,
\end{equation}
so that
\begin{align}
\phi^*+\sum_{m=0}^{r-1}\phi_m=&\, \frac{s-r}{s}+(s-r)\sum_{m=0}^{r-1}
(s_{m+1}^{-1}-s_m^{-1})\notag \\
=&\,\frac{s-r}{s}+(s-r)\left( \frac{1}{s-r}-\frac{1}{s}\right)=1.\label{7.3}
\end{align}

\begin{lemma}\label{lemma7.1}
Suppose that $a$ and $b$ are integers with $0\le a<b\le \tet^{-1}$ and $(k-r+1)b\ge ra$, and suppose 
further that $a\le b/\sqrt{k}$. Then whenever $a'$ is an integer with $a'\ge a$ for which 
$(k-r+1)b\ge ra'$, one has
$$K_{a,b}^{r,r}(X)\ll \left(J_{s+r}(X/M^b)\right)^{\phi^*}\prod_{m=0}^{r-1}
\left( (M^{b_m-a})^sI_{b,b_m}^{r,r}(X)\right)^{\phi_m},$$
where we write $b_m=(k-m)b-ma'$.
\end{lemma}

\begin{proof} We prove by induction that for $0\le l\le r-1$, one has
\begin{equation}\label{7.4}
K_{a,b}^{r,r}(X)\ll \left( K_{a,b}^{l,r}(X)\right)^{\phi^*_l}\prod_{m=l}^{r-1}
\left( (M^{b_m-a})^sI_{b,b_m}^{r,r}(X)\right)^{\phi_m},
\end{equation}
where
$$\phi^*_l=(s-r)/(s-l).$$
The conclusion of the lemma follows from the case $l=0$ of (\ref{7.4}), on noting that Lemma 
\ref{lemma2.1} delivers the estimate $K_{a,b}^{0,r}(X)\ll J_{s+r}(X/M^b)$.\par

We observe first that the inductive hypothesis (\ref{7.4}) holds when $l=r-1$, as a consequence of the 
case $m=r-1$ of Lemma \ref{lemma6.1} and the definitions (\ref{7.1}) and (\ref{7.2}). Suppose then that 
$J$ is a positive integer not exceeding $r-2$, and that the inductive hypothesis (\ref{7.4}) holds for 
$J<l\le r-1$. An application of Lemma \ref{lemma6.1} yields the estimate
$$K_{a,b}^{J+1,r}(X)\ll \left( (M^{b_J-a})^sI_{b,b_J}^{r,r}(X)\right)^{1/(s-J)}
\left( K_{a,b}^{J,r}(X)\right)^{(s-J-1)/(s-J)}.$$
On substituting this bound into the estimate (\ref{7.4}) with $l=J+1$, one obtains the new upper bound
$$K_{a,b}^{r,r}(X)\ll \left( K_{a,b}^{J,r}(X)\right)^{\phi^*_J}\prod_{m=J}^{r-1}
\left( (M^{b_m-a})^sI_{b,b_m}^{r,r}(X)\right)^{\phi_m},$$
and thus the inductive hypothesis holds with $l=J$. This completes the inductive step, so in view of our 
earlier remarks, the conclusion of the lemma now follows.
\end{proof}

We next convert Lemma \ref{lemma7.1} into a more portable form by making use of the anticipated 
magnitude operator $\llbracket \,\cdot \,\rrbracket$ introduced in (\ref{2.21}) to (\ref{2.23}).

\begin{lemma}\label{lemma7.2} Suppose that $a$ and $b$ are integers with $0\le a<b\le \tet^{-1}$ and 
$(k-r+1)b\ge ra$, and suppose further that $a\le b/\sqrt{k}$. Then whenever $a'$ is an integer with 
$a'\ge a$ for which $(k-r+1)b\ge ra'$, one has
$$\llbracket K_{a,b}^{r,r}(X)\rrbracket \ll \left( (X/M^b)^{\Lam+\del}\right)^{\phi^*}
\prod_{m=0}^{r-1}\llbracket I_{b,b_m}^{r,r}(X)\rrbracket^{\phi_m},$$
where $b_m=(k-m)b-ma'$.
\end{lemma}

\begin{proof} We find from Lemma \ref{lemma7.1} in combination with (\ref{2.23}) that
\begin{equation}\label{7.5}
\llbracket K_{a,b}^{r,r}(X)\rrbracket \ll M^{\mu^*b+\nu^*a}
\left( (X/M^b)^{\Lam+\del}\right)^{\phi^*}\prod_{m=0}^r
\llbracket I_{b,b_m}^{r,r}(X)\rrbracket^{\phi_m},
\end{equation}
where
$$\mu^*=s-\phi^*(s+r)-\sum_{m=0}^{r-1}r\phi_m\quad \text{and}\quad 
\nu^*=r-s\sum_{m=0}^{r-1}\phi_m.$$
On recalling (\ref{7.2}) and (\ref{7.3}), a modicum of computation confirms that
$$\mu^*=s-\phi^*(s+r)-r(1-\phi^*)=0\quad \text{and}\quad \nu^*=r-s(1-\phi^*)=0.$$
The conclusion of the lemma is therefore immediate from (\ref{7.5}).
\end{proof}

We next turn to the task of establishing a multistep multigrade combination. Here, in order to keep 
complications under control, we discard some information not essential to our ultimate iterative process. 
We begin by introducing some additional notation. We recall that $R$ is a positive integer sufficiently large 
in terms of $s$ and $k$. We consider $R$-tuples of integers $(m_1,\ldots ,m_R)\in [0,r-1]^R$, to each of 
which we associate an $R$-tuple of integers $\bfh=(h_1(\bfm),\ldots ,h_R(\bfm))\in [0,\infty)^R$. The 
integral tuples $\bfh(\bfm)$ will be fixed as the iteration proceeds, with $h_n(\bfm)$ depending at most 
on the first $n$ coordinates of $(m_1,\ldots ,m_R)$. We may abuse notation in some circumstances by 
writing $h_n(\bfm,m_n)$ or $h_n(m_1,\ldots ,m_{n-1},m_n)$ in place of $h_n(m_1,\ldots ,m_R)$, 
reflecting the latter implicit dependence. We suppose that a (large) positive integer $b$ has already been 
fixed. We then define the sequences $(a_n)=(a_n(\bfm;\bfh))$ and $(b_n)=(b_n(\bfm;\bfh))$ by putting
\begin{equation}\label{7.6}
a_0=\lfloor b/\sqrt{k}\rfloor\quad \text{and}\quad b_0=b,
\end{equation}
and then applying the iterative relations
\begin{equation}\label{7.7}
a_n=b_{n-1}\quad \text{and}\quad b_n=(k-m_n)b_{n-1}-m_na_{n-1}+h_n(\bfm)\quad (1\le n\le R).
\end{equation}
Finally, we define the quantity $\Tet_n(\bfm;\bfh)$ for $0\le n\le R$ by putting
\begin{equation}\label{7.8}
\Tet_n(\bfm;\bfh)=(X/M^b)^{-\Lam-\del}\llbracket K_{a_n,b_n}^{r,r}(X)\rrbracket +M^{-3sk^Rb}.
\end{equation}

\begin{lemma}\label{lemma7.3}
Suppose that $a$ and $b$ are integers with $0\le a<b\le (16Rk^{2R}\tet)^{-1}$ and $(k-r+1)b\ge ra$, 
and suppose further that $a\le b/\sqrt{k}$. Then there exists a choice for $\bfh(\bfm)\in [0,r-1]^R$, 
satisfying the condition that $0\le h_n(\bfm)\le 15k^Rb$ $(1\le n\le R)$, and for which one has
$$(X/M^b)^{-\Lam-\del}\llbracket K_{a,b}^{r,r}(X)\rrbracket \ll \prod_{\bfm\in [0,r-1]^R}
\Tet_R(\bfm;\bfh)^{\phi_{m_1}\ldots \phi_{m_R}}.$$
\end{lemma}

\begin{proof} We prove by induction on $l$ that when $1\le l\le R$, one has the upper bound
\begin{equation}\label{7.9}
(X/M^b)^{-\Lam-\del}\llbracket K_{a,b}^{r,r}(X)\rrbracket \ll \prod_{\bfm\in [0,r-1]^l}
\Tet_l(\bfm;\bfh)^{\phi_{m_1}\ldots \phi_{m_l}}.
\end{equation}
We observe first that, as a consequence of (\ref{7.3}) and Lemma \ref{lemma7.2}, one has
\begin{equation}\label{7.10}
(X/M^b)^{-\Lam-\del}\llbracket K_{a,b}^{r,r}(X)\rrbracket \ll \prod_{m=0}^{r-1}
\left( (X/M^b)^{-\Lam-\del}\llbracket I_{b,b_m^*}^{r,r}(X)\rrbracket \right)^{\phi_m},
\end{equation}
where
$$b_m^*=(k-m)b-m\lfloor b/\sqrt{k}\rfloor .$$
Moreover, it follows from Lemma \ref{lemma4.2} that for each $m$ with $0\le m\le r-1$, there exists an 
integer $h=h(m)$, with $0\le h\le 15k^Rb$, having the property that
$$I_{b,b_m^*}^{r,r}(X)\ll (M^h)^{2s/3}K_{b,b_m^*+h}^{r,r}(X)+
(M^{15k^Rb})^{-s/4}(X/M^{b_m^*})^s(X/M^b)^{\lam-s},$$
whence from (\ref{2.22}) and (\ref{2.23}) we discern that
\begin{equation}\label{7.11}
\llbracket I_{b,b_m^*}^{r,r}(X)\rrbracket \ll M^{-hs/3}\llbracket K_{b,b_1}^{r,r}(X)\rrbracket 
+M^{-3sk^Rb}(X/M^b)^{\Lam+\del},
\end{equation}
where
$$b_1=(k-m)b_0-ma_0+h(m).$$
The inductive hypothesis (\ref{7.9}) therefore follows in the case $R=1$ by substituting (\ref{7.11}) into 
(\ref{7.10}). Notice here that we have discarded the factor $M^{-hs/3}$ from (\ref{7.11}), since at this 
stage of our argument, it serves only as an unwelcome complication.\par

Suppose next that the inductive hypothesis (\ref{7.9}) holds for $1\le l<L$, with some integer $L$ 
satisfying $L\le R$. Consider the quantity $\Tet_{L-1}(\bfm;\bfh)$ for a given tuple 
$\bfm\in [0,r-1]^{L-1}$ and a fixed tuple $\bfh=\bfh(\bfm)$. We distinguish two possibilities. If it is the 
case that
$$\Tet_{L-1}(\bfm;\bfh)\ll M^{-3sk^Rb},$$
then from (\ref{7.3}) one finds that
$$\Tet_{L-1}(\bfm;\bfh)=\left( \prod_{m_L=0}^{r-1}(M^{-3sk^Rb})^{\phi_{m_L}}\right)^{s/r},$$
so that
\begin{equation}\label{7.12}
\Tet_{L-1}(\bfm;\bfh)\ll \prod_{m_L=0}^{r-1}(M^{-3sk^Rb})^{\phi_{m_L}}\ll \prod_{m_L=0}^{r-1}
\Tet_L(\bfm,m_L;\bfh,0)^{\phi_{m_L}}.
\end{equation}

\par When
$$\Tet_{L-1}(\bfm;\bfh)\ll (X/M^b)^{-\Lam-\del}\llbracket K_{a_{L-1},b_{L-1}}^{r,r}(X)\rrbracket ,$$
meanwhile, we apply Lemma \ref{lemma7.2} to obtain the bound
\begin{equation}\label{7.13}
\Tet_{L-1}(\bfm;\bfh)\ll \left(\frac{X}{M^b}\right)^{-\Lam-\del}
\left(\frac{X}{M^{b_{L-1}}}\right)^{(\Lam+\del)(1-r/s)}\prod_{m_L=0}^{r-1} 
\llbracket I_{b_{L-1},b_{m_L}^*}^{r,r}(X)\rrbracket^{\phi_{m_L}},
\end{equation}
where
$$b_{m_L}^*=(k-m_L)b_{L-1}-ma_{L-1}.$$
Notice here that the hypothesis $b\le (16Rk^{2R}\tet)^{-1}$ ensures that $b_{L-1}\le \tet^{-1}$, so that 
Lemma \ref{lemma7.2} is applicable. Again invoking Lemma \ref{lemma4.2}, we find that for each integer
 $m_L$ with $0\le m_L\le r-1$, there exists an integer $h_L=h_L(\bfm,m_L)$, with 
$0\le h_L\le 15k^Rb$, having the property that
$$\llbracket I_{b_{L-1},b_{m_L}^*}^{r,r}(X)\rrbracket \ll M^{-h_Ls/3}
\llbracket K_{b_{L-1},b_L}^{r,r}(X)\rrbracket +M^{-3sk^Rb}(X/M^{b_{L-1}})^{\Lam+\del}.$$
Thus we deduce that
\begin{align*}
(X/M^b)^{-\Lam-\del}\llbracket I_{b_{L-1},b_{m_L}^*}^{r,r}(X)\rrbracket &\ll 
(X/M^b)^{-\Lam-\del}\llbracket K_{a_L,b_L}^{r,r}(X)\rrbracket +M^{-3sk^Rb}\\
&\ll \Tet_L(\bfm,m_L;\bfh,h_L).
\end{align*}
On substituting this estimate into (\ref{7.13}), we deduce  from (\ref{7.3}) that
\begin{align}
\Tet_{L-1}(\bfm;\bfh)&\ll (M^{b-b_{L-1}})^{(1-r/s)(\Lam+\del)}
\prod_{m_L=0}^{r-1}\Tet_L(\bfm,m_L;\bfh,h_L)^{\phi_{m_L}}\notag \\
&\ll \prod_{m_L=0}^{r-1}\Tet_L(\bfm,m_L;\bfh,h_L)^{\phi_{m_L}}.\label{7.14}
\end{align}
Applying this estimate in combination with (\ref{7.12}) within the case $l=L-1$ of the inductive hypothesis 
(\ref{7.9}), we conclude that for some choice of the integer $h_L=h_L(\bfm)$, one has the upper bound
$$(X/M^b)^{-\Lam-\del}\llbracket K_{a,b}(X)\rrbracket \ll 
\prod_{\bfm\in [0,r-1]^{L-1}}\prod_{m_L=0}^{r-1}
\Tet_L(\bfm,m_L;\bfh,h_L)^{\phi_{m_1}\ldots \phi_{m_L}}.$$
This confirms the inductive hypothesis (\ref{7.9}) for $l=L$, and thus the conclusion of the lemma follows 
by induction.
\end{proof}

We remark that in obtaining the estimate (\ref{7.14}), we made use of the trivial lower bound 
$b_{L-1}\ge b$. By discarding the power $M^{b-b_{L-1}}$ at this point, we are throwing away potentially 
useful information. However, it transpires that the weak information made available by the factor $M^b$ in 
the conclusion of Lemma \ref{lemma7.3} already suffices for our purposes in the main iteration.

\section{The latent monograde process} Our objective in this section is to convert the block estimate 
encoded in Lemma \ref{lemma7.3} into a single monograde estimate that can be incorporated into our 
iterative method. We begin by recalling an elementary lemma from our previous work \cite{Woo2014}.

\begin{lemma}\label{lemma8.1}
Suppose that $z_0,\ldots ,z_l\in \dbC$, and that $\bet_i$ and $\gam_i$ are positive real numbers for 
$0\le i\le l$. Put $\Ome=\bet_0\gam_0+\ldots +\bet_l\gam_l$. Then one has
$$|z_0^{\bet_0}\ldots z_l^{\bet_l}|\le \sum_{i=0}^l|z_i|^{\Ome/\gam_i}.$$
\end{lemma}

\begin{proof} This is \cite[Lemma 8.1]{Woo2014}.
\end{proof}

Before proceeding further, we introduce some additional notation. Define the positive number $s_0$ by 
means of the relation
\begin{equation}\label{8.1}
s_0^R=\frac{\tet_+^{R+1}-\tet_-^{R+1}}{\tet_+-\tet_-}-\frac{\tet_+\tet_-}{r\sqrt{k}}
\left( \frac{\tet_+^R-\tet_-^R}{\tet_+-\tet_-}\right) ,
\end{equation}
in which $\tet_\pm$ are defined as in (\ref{2.6}). We recall that, in view of (\ref{2.7}), one has 
$s<s_0$. Next, casting an eye toward the iterative relations (\ref{7.6}) and (\ref{7.7}), we define 
new sequences $(\atil_n)=(\atil_n(\bfm))$ and $(\btil_n)=(\btil_n(\bfm))$ by means of the relations
\begin{equation}\label{8.2}
\atil_0=1/\sqrt{k}\quad \text{and}\quad \btil_0=1,
\end{equation}
and
\begin{equation}\label{8.3}
\atil_n=\btil_{n-1}\quad \text{and}\quad \btil_n=(k-m_n)\btil_{n-1}-m_n\atil_{n-1}\quad (1\le n\le R).
\end{equation}
We then define
\begin{equation}\label{8.4}
k_\bfm=\btil_R(\bfm)\quad \text{and}\quad \rho_\bfm=\btil_R(\bfm)(s/s_0)^R\quad \text{for}\quad 
\bfm\in [0,r-1]^R.
\end{equation}
The motivation for defining the sequences $(\atil_n)$ and $(\btil_n)$ is to provide a base pair of sequences 
corresponding to the simplified situation in which $h_n(\bfm)=0$ for all $n$ and $\bfm$. It transpires that 
the corresponding sequences $(a_n)$ and $(b_n)$, equipped with potentially positive values of $h_n(\bfm)$, 
may be bounded below by the sequences $(\atil_n)$ and $(\btil_n)$, although this turns out to be less simple 
to establish than might be supposed.

\begin{lemma}\label{lemma8.2} Suppose that $\Lam\ge 0$, and let $a$ and $b$ be integers with 
$$0\le a<b\le (20Rk^{2R}\tet)^{-1}\quad \text{and}\quad (k-r+1)b\ge ra,$$
and suppose further that $a\le b/\sqrt{k}$. Suppose in addition that there are real numbers $\psi$, $c$ and $\gam$, with
$$0\le c\le (2\del)^{-1}\tet,\quad \gam\ge -sb\quad \text{and}\quad \psi\ge 0,$$
such that
\begin{equation}\label{8.5}
X^\Lam M^{\Lam \psi}\ll X^{c\del}M^{-\gam}\llbracket K_{a,b}^{r,r}(X)\rrbracket .
\end{equation}
Then, for some $\bfm\in [0,r-1]^R$, there is a real number $h$ with $0\le h\le 16k^{2R}b$, and positive 
integers $a'$ and $b'$ with $a'\le b'/\sqrt{k}$, such that
\begin{equation}\label{8.6}
X^\Lam M^{\Lam \psi'}\ll X^{c'\del}M^{-\gam'}\llbracket K_{a',b'}^{r,r}(X)\rrbracket ,
\end{equation}
where $\psi'$, $c'$, $\gam'$ and $b'$ are real numbers satisfying the conditions
$$\psi'=\rho_\bfm\left(\psi+\left(1-r/s\right)b\right),\quad c'=\rho_\bfm(c+1),\quad 
\gam'=\rho_\bfm\gam,\quad b'=k_\bfm b+h.$$
Moreover, the real number $k_\bfm$ satisfies $2^R\le k_\bfm\le k^R$.
\end{lemma}

\begin{proof} We deduce from the postulated bound (\ref{8.5}) and Lemma \ref{lemma7.3} that 
there exists a choice of the tuple $\bfh=\bfh(\bfm)$, with $0\le h_n(\bfm)\le 15k^Rb$ 
$(1\le n\le R)$, such that
$$X^\Lam M^{\Lam \psi}\ll X^{(c+1)\del}M^{-\gam}(X/M^b)^\Lam \prod_{\bfm\in [0,r-1]^R}
\Tet_R(\bfm;\bfh)^{\phi_{m_1}\ldots \phi_{m_R}}.$$
Consequently, one has
$$\prod_{\bfm\in [0,r-1]^R}\Tet_R(\bfm;\bfh)^{\phi_{m_1}\ldots \phi_{m_R}}\gg 
X^{-(c+1)\del}M^{\Lam (\psi+b)+\gam}.$$
Note that, in view of (\ref{7.3}), one has
\begin{equation}\label{8.7}
\sum_{m=0}^{r-1}\phi_m=r/s,
\end{equation}
so that
$$\sum_{\bfm\in [0,r-1]^R}\phi_{m_1}\ldots \phi_{m_R}=(r/s)^R\le r/s.$$
Then we deduce from the definition (\ref{7.8}) of $\Tet_n(\bfm;\bfh)$ that
\begin{align}\label{8.8}
\prod_{\bfm\in [0,r-1]^R}&\left( X^{-\Lam}\llbracket K_{a_R,b_R}^{r,r}(X)\rrbracket 
+M^{-3sk^Rb}\right)^{\phi_{m_1}\ldots \phi_{m_R}}\notag \\
&\ \ \ \ \ \ \ \ \ \ \ \ \ \ \gg X^{-(c+1)\del}M^{\Lam (\psi+(1-r/s)b)+\gam}.
\end{align}

\par In preparation for our application of Lemma \ref{lemma8.1}, we examine the exponents 
$\phi_{m_1}\ldots \phi_{m_R}$ occurring in the lower bound (\ref{8.8}). Our plan is to 
apply Lemma \ref{lemma8.1} with exponents given by
$$\bet_\bfm^{(n)}=\phi_{m_1}\ldots \phi_{m_n}\quad \text{and}\quad 
\gam_\bfm^{(n)}=\btil_n(\bfm)\quad (\bfm\in [0,r-1]^n),$$
in the natural sense. In order to analyse the quantity $\Ome$ that will emerge from the application 
of this lemma, we define 
$$B_n=\sum_{\bfm\in [0,r-1]^n}\bet_\bfm^{(n)}\btil_n(\bfm)\quad \text{and}\quad 
A_n=\sum_{\bfm\in [0,r-1]^n}\bet_\bfm^{(n)}\atil_n(\bfm),$$
and then put $\Ome=B_R$. From the iterative formulae (\ref{8.2}) and (\ref{8.3}), we obtain for 
$n\ge 1$ the relations
\begin{align}
B_{n+1}&=\sum_{m_{n+1}=0}^{r-1}\sum_{\bfm\in [0,r-1]^n}\left( (k-m_{n+1})\btil_n(\bfm)-
m_{n+1}\atil_n(\bfm)\right)\phi_{m_1}\ldots \phi_{m_{n+1}}\notag \\
&=B_n\sum_{m=0}^{r-1}(k-m)\phi_m-A_n\sum_{m=0}^{r-1}m\phi_m,\label{8.9}
\end{align}
and
\begin{equation}\label{8.10}
A_{n+1}=\sum_{m_{n+1}=0}^{r-1}
\sum_{\bfm\in [0,r-1]^n}\btil_n(\bfm)\phi_{m_1}\ldots \phi_{m_{n+1}}=
B_n\sum_{m=0}^{r-1}\phi_m.
\end{equation}

\par We observe that from (\ref{7.1}), one has
$$\sum_{m=0}^{r-1}(k-m)\phi_m=\sum_{m=0}^{r-1}\frac{(k-m)(s-r)}{(s-m)(s-m-1)}.$$
The summands on the right hand side here may be rewritten in the shape
$$\frac{(r-m)(k-m)}{s-m}-\frac{(r-m-1)(k-m-1)}{s-m-1}-\frac{r-m-1}{s-m-1}.$$
Hence we obtain
$$s\sum_{m=0}^{r-1}(k-m)\phi_m=kr-s\sum_{l=1}^r\frac{r-l}{s-l}=kr-\tfrac{1}{2}r(r-1)-\Del,$$
in which $\Del$ is defined via (\ref{2.3}). In addition, one has
\begin{align*}
s\sum_{m=0}^{r-1}m\phi_m&=sk\sum_{m=0}^{r-1}\phi_m-s\sum_{m=0}^{r-1}(k-m)\phi_m\\
&=kr-(kr-\tfrac{1}{2}r(r-1)-\Del)\\
&=\tfrac{1}{2}r(r-1)+\Del .
\end{align*}

\par Utilising the formulae just obtained, we conclude from (\ref{8.9}) that for $n\ge 1$, one has
$$sB_{n+1}=(kr-\tfrac{1}{2}r(r-1)-\Del)B_n-(\tfrac{1}{2}r(r-1)+\Del)A_n,$$
while from (\ref{8.7}) and (\ref{8.10}), one sees that
$$sA_{n+1}=rB_n.$$
Then, on recalling (\ref{2.5}), we arrive at the iterative relation
\begin{equation}\label{8.11}
s^2B_{n+2}-s\gra B_{n+1}+\grb B_n=0\quad (n\ge 1).
\end{equation}
In addition, also from (\ref{2.5}) and (\ref{8.2}), one has the initial data
\begin{equation}\label{8.12}
sB_1=s\sum_{m=0}^{r-1}\left( (k-m)\btil_0-m\atil_0\right)\phi_m=\gra-\grb/(r\sqrt{k}),
\end{equation}
and from (\ref{8.2}) and (\ref{8.7}), one finds that
$$sA_1=s\sum_{m=0}^{r-1}\btil_0\phi_m=r.$$
Thus we see also that
\begin{align}
s^2B_2&=s^2\sum_{m=0}^{r-1}\left( (k-m)B_1-mA_1\right)\phi_m=s(\gra B_1-\grb A_1/r)\notag \\
&=\gra(\gra-\grb/(r\sqrt{k}))-\grb .\label{8.13}
\end{align}

\par On recalling (\ref{2.6}), one finds that the recurrence formula (\ref{8.11}) has a solution of the shape
$$s^nB_n=\sig_+\tet_+^n+\sig_-\tet_-^n\quad (n\ge 1),$$
where, in view of (\ref{8.12}) and (\ref{8.13}), one has
$$\sig_+\tet_++\sig_-\tet_-=sB_1=\gra -\grb/(r\sqrt{k})$$
and
$$\sig_+\tet_+^2+\sig_-\tet_-^2=s^2B_2=\gra(\gra-\grb/(r\sqrt{k}))-\grb.$$
Since $\gra=\tet_++\tet_-$ and $\grb=\tet_+\tet_-$, we therefore deduce that
$$s^nB_n=\frac{\tet_+^{n+1}-\tet_-^{n+1}}{\tet_+-\tet_-}-\frac{\tet_+\tet_-}{r\sqrt{k}}
\left( \frac{\tet_+^n-\tet_-^n}{\tet_+-\tet_-}\right) .$$
In particular, on recalling (\ref{8.1}), we find that $s^RB_R=s_0^R$, so that
\begin{equation}\label{8.14}
B_R=(s_0/s)^R.
\end{equation}
Also, therefore, it follows from (\ref{2.7}) that $s<s_0$, and hence also that $B_R>1$.\par

Returning now to the application of Lemma \ref{lemma8.1}, we note first that $\Ome=B_R$, and hence 
(\ref{8.8}) yields the relation
$$\sum_{\bfm\in [0,r-1]^R}\left( X^{-\Lam}\llbracket 
K_{a_R,b_R}^{r,r}(X)\rrbracket +M^{-3sk^Rb}\right)^{B_R/\btil_R(\bfm)}\gg 
X^{-(c+1)\del}M^{\Lam (\psi+(1-r/s)b)+\gam}.$$
But in view of (\ref{8.4}) and (\ref{8.14}), one has $\btil_R(\bfm)/B_R=\rho_\bfm$, and thus we find 
that for some tuple $\bfm\in [0,r-1]^R$, one has
$$X^{-\Lam}\llbracket K_{a_R,b_R}^{r,r}(X)\rrbracket +M^{-3sk^Rb}\gg X^{-\rho_\bfm(c+1)\del}
M^{\Lam \rho_\bfm(\psi+(1-r/s)b)+\rho_\bfm\gam},$$
whence,
\begin{equation}\label{8.15}
X^{-\Lam}\llbracket K_{a_R,b_R}^{r,r}(X)\rrbracket +M^{-3sk^Rb}\gg X^{-c'\del}
M^{\Lam \psi'+\gam'}.
\end{equation}

\par Two further issues remain to be resolved, the first being the removal of the term $M^{-3sk^Rb}$ on 
the left hand side of (\ref{8.15}). We observe that the relations (\ref{8.3}) ensure that 
$\btil_R(\bfm)\le k^R$, and hence (\ref{8.4}) reveals that $\rho_\bfm<\btil_R(\bfm)\le k^R$. Our 
hypothesis on $c$ therefore ensures that $c'\del<2k^Rc\del\le k^R\tet$, so that
$$X^{-c'\del}M^{\Lam \psi'+\gam'}\ge M^{-k^R+\rho_\bfm \gam}> M^{-k^R-sk^Rb}.$$
Since
$$M^{-3sk^Rb}\le M^{-2sk^R-sk^Rb},$$
it follows from (\ref{8.15}) that
\begin{equation}\label{8.16}
X^{-\Lam}\llbracket K_{a_R,b_R}^{r,r}(X)\rrbracket \gg X^{-c'\del}M^{\Lam \psi'+\gam'}.
\end{equation}

\par Our final task consists of extracting appropriate constraints on the parameters $a_R$ and $b_R$. We 
note first that from (\ref{7.6}) one has $a_0=\lfloor b/\sqrt{k}\rfloor\le b/\sqrt{k}$, so that on writing
$$h_1^*(\bfm)=h_1(\bfm)+m_1(b/\sqrt{k}-a_0)$$
and
$$h_n^*(\bfm)=h_n(\bfm)\quad (n>1),$$
we have
$$h_1^*(\bfm)\le h_1(\bfm)+m_1\le h_1(\bfm)+k<16k^Rb$$
and
$$h_n^*(\bfm)=h_n(\bfm)\le 16k^Rb\quad (n>1).$$
Define the sequences $(a_n^*)=(a_n^*(\bfm;\bfh))$ and $(b_n^*)=(b_n^*(\bfm;\bfh))$ by means of 
the relations
$$a_n^*=b_{n-1}^*\quad \text{and}\quad b_n^*=(k-m_n)b_{n-1}^*-m_na_{n-1}^*+h_n^*(\bfm)
\quad (1\le n\le R),$$
where
$$a_0^*=b/\sqrt{k}\quad \text{and}\quad b_0^*=b.$$
Then a comparison with (\ref{7.6}) and (\ref{7.7}) reveals that $a_n^*=a_n$ and $b_n^*=b_n$ for 
$n\ge 1$. Moreover, it is apparent from (\ref{8.2}) and (\ref{8.3}) that $\atil_n b=a_n^*(\bfm;{\bf0})$ 
and $\btil_nb=b_n^*(\bfm;\bf0)$ for $n\ge 0$. We claim that for $n\ge 0$, one has
\begin{equation}\label{8.17}
b_n^*(\bfm;{\bf0})\le b_n^*(\bfm;\bfh)\le b_n^*(\bfm;{\bf0})+16k^{R+n}b.
\end{equation}
The validity of these inequalities for $n=0$ is immediate from the definition of $b_n^*(\bfm;\bfh)$. We 
use an inductive argument to establish (\ref{8.17}) for $n\ge 1$, though this requires some discussion.
\par

Observe first that, by linearity, the recurrence sequence $b_n^*(\bfm;\bfh)$ is given by the formula
\begin{equation}\label{8.18}
b_n^*(\bfm;\bfh)=b_n^*(\bfm;{\bf0})+\sum_{l=1}^Rc_{n,l}(\bfm;\bfh)\quad (0\le n\le R),
\end{equation}
in which $c_{n,l}(\bfm;\bfh)$ is determined by the relations
$$c_{l,l}(\bfm;\bfh)=h_l^*(\bfm)\quad \text{and}\quad c_{n,l}(\bfm;\bfh)=0\quad (n<l),$$
with
\begin{equation}\label{8.19}
c_{n+1,l}(\bfm;\bfh)=(k-m_{n+1})c_{n,l}(\bfm;\bfh)-m_{n+1}c_{n-1,l}(\bfm;\bfh)\quad (n>l).
\end{equation}
We recall our assumption that $r\le k-\lceil 2\sqrt{k}\rceil +2$, which ensures that
$$m_{n+1}\le r-1\le k-\lceil 2\sqrt{k}\rceil +1.$$
In view of this upper bound, we are able to prove by induction that
\begin{equation}\label{8.20}
c_{n+1,l}(\bfm;\bfh)>\sqrt{k}c_{n,l}(\bfm;\bfh)
\end{equation}
for each $n$. In order to confirm this assertion, note first that such is the case when $n=l-1$. Granted that 
$c_{n,l}(\bfm;\bfh)>\sqrt{k}c_{n-1,l}(\bfm;\bfh)$, meanwhile, one deduces from (\ref{8.19}) that
\begin{align}
c_{n+1,l}(\bfm;\bfh)&\ge (\lceil 2\sqrt{k}\rceil -1)c_{n,l}(\bfm;\bfh)-
(k-\lceil 2\sqrt{k}\rceil +1)c_{n,l}(\bfm;\bfh)/\sqrt{k}\notag \\
&\ge \left( (2\sqrt{k}-1)-(\sqrt{k}-2+1/\sqrt{k})\right) c_{n,l}(\bfm;\bfh)\notag \\
&>\sqrt{k}c_{n,l}(\bfm;\bfh),\label{8.21}
\end{align}
establishing this inductive hypothesis. Thus, in particular, one arrives at the lower bound 
$c_{n,l}(\bfm;\bfh)\ge 0$ for all $n$. On substituting this conclusion into (\ref{8.18}), we deduce that 
$b_n^*(\bfm;\bfh)\ge b_n^*(\bfm;{\bf0})$ for every $n$, confirming the first of the inequalities in 
(\ref{8.17}).\par

Next, making use of the lower bound $c_{n,l}(\bfm;\bfh)\ge 0$, just obtained, within (\ref{8.19}), 
we find that
$$c_{n+1,l}(\bfm;\bfh)\le kc_{n,l}(\bfm;\bfh)\quad (n>l),$$
so that
$$c_{n,l}(\bfm;\bfh)\le k^{n-l}h_l^*(\bfm)\quad (n\ge l).$$
We therefore deduce from (\ref{8.18}) and the bound $h_n^*(\bfm)\le 16k^Rb$ that
$$b_n^*(\bfm;\bfh)\le b_n^*(\bfm;{\bf0})+16k^Rb\sum_{l=1}^n k^{n-l}\le b_n^*(\bfm;{\bf0})+
16k^{R+n}b.$$
This confirms the second of the inequalities in (\ref{8.17}).\par

We now make use of the inequalities (\ref{8.17}). Observe first that in view of (\ref{8.4}), one has
$$b_R=b_R^*(\bfm;\bfh)\le b_R^*(\bfm;{\bf0})+16k^{2R}b=\btil_R(\bfm)b+16k^{2R}b=
k_\bfm b+16k^{2R}b.$$
In addition, one has
$$b_R=b_R^*(\bfm;\bfh)\ge b_R^*(\bfm;{\bf0})=\btil_R(\bfm)b=k_\bfm b.$$
Thus, there exists an integer $h_\bfm$, with $0\le h_\bfm\le 16k^{2R}b$, for which one has 
$b_R=k_\bfm b+h_\bfm$.\par

The argument confirming (\ref{8.21}) shows also that for each $n\ge 0$, one has
\begin{equation}\label{8.22}
b_{n+1}^*(\bfm;{\bf0})>\sqrt{k}b_n^*(\bfm;{\bf0}).
\end{equation}
Then it follows from (\ref{8.18}) and (\ref{8.20}) that
$$b_R^*(\bfm;\bfh)>\sqrt{k}b_{R-1}^*(\bfm;\bfh),$$
whence
$$a_R=b_{R-1}=b_{R-1}^*(\bfm;\bfh)<b_R^*(\bfm;\bfh)/\sqrt{k}=b_R/\sqrt{k}.$$
Then we deduce from (\ref{8.16}) that there exist integers $b'=b_R$, $a'=a_R$ and $h=h_\bfm$ satisfying 
the conditions
$$0\le h_\bfm\le 16k^{2R}b,\quad b'=k_\bfm b+h\quad \text{and}\quad a'\le b'/\sqrt{k},$$
and for which
$$X^{-\Lam }\llbracket K_{a',b'}^{r,r}(X)\rrbracket \gg X^{-c'\del}M^{\Lam \psi'+\gam'}.$$
The desired conclusion (\ref{8.6}) now follows, with all the associated conditions.\par

It now remains only to confirm the bounds on $k_\bfm$ asserted in the final line of the statement of the 
lemma. For this, we note that by applying (\ref{8.22}) in an inductive argument, one obtains from 
(\ref{8.4}) the lower bound
$$k_\bfm=\btil_R(\bfm)=b_R^*(\bfm;{\bf0})/b\ge (\sqrt{k})^Rb_0^*(\bfm;{\bf0})/b\ge 2^R.$$
Likewise, though more directly, one has
$$k_\bfm=\btil_R(\bfm)=b_R^*(\bfm;{\bf0})/b\le k^Rb_0^*(\bfm;{\bf0})/b=k^R.$$
Thus $2^R\le k_\bfm\le k^R$, and the proof of the lemma is complete.
\end{proof}

\section{The iterative process} We begin with a crude estimate of use at the conclusion of our argument.

\begin{lemma}\label{lemma9.1}
Suppose that $a$ and $b$ are integers with $0\le a<b\le (2\tet)^{-1}$. Then provided that $\Lam\ge 0$, 
one has
$$\llbracket K_{a,b}^{r,r}(X)\rrbracket \ll X^{\Lam+\del}.$$
\end{lemma}

\begin{proof} On considering the underlying Diophantine equations, we deduce from Lemma 
\ref{lemma2.2} that
$$K_{a,b}^{r,r}(X)\ll \left( J_{s+r}(X/M^a)\right)^{r/(s+r)}\left( J_{s+r}(X/M^b)\right)^{s/(s+r)},$$
whence
$$\llbracket K_{a,b}^{r,r}(X)\rrbracket \ll \frac{X^\del 
\left( (X/M^a)^{r/(s+r)}(X/M^b)^{s/(s+r)}\right)^{s+r+\Lam}}{(X/M^a)^r(X/M^b)^s}
\le X^{\Lam+\del}.$$
This completes the proof of the lemma.
\end{proof}

We now come to the main event.

\begin{theorem}\label{theorem9.2} Suppose that $s$, $k$ and $r$ are natural numbers with $k\ge 4$, 
$$2\le r\le k-\lceil 2\sqrt{k}\rceil +2\quad \text{and}\quad \max\{2r-1,\tfrac{1}{2}r(r-1)+\Del\}\le s<s_1,$$
where
$$s_1=\tfrac{1}{2}(kr-\tfrac{1}{2}r(r-1)-\Del)
\left( 1+\sqrt{1-\frac{4r(\tfrac{1}{2}r(r-1)+\Del)}{(kr-\tfrac{1}{2}r(r-1)-\Del)^2}}\right) ,$$
and
$$\Del=\sum_{m=1}^{r-1}\frac{m(r-m)}{s-m}.$$
Then for each $\eps>0$, one has
$$J_{s+r}(X)\ll X^{s+r+\eps}.$$
\end{theorem}

\begin{proof} We prove that when $s$ is the largest integer smaller than $s_1$, then one has 
$\Lam\le 0$, for in such circumstances the conclusion of the lemma follows at once from 
(\ref{2.24}). By reference to (\ref{2.5}) and (\ref{2.6}), we find that $s<s_1=\tet_+$. Thus, 
since $R$ has been chosen sufficiently large in terms of $s$, $k$ and $\tet_+$, the hypothesis 
$s<s_1$ ensures that $s<s_0$, where $s_0$ is defined by means of (\ref{8.1}). Assume then that 
$\Lam\ge 0$, for otherwise there is 
nothing to prove. We begin by noting that as a consequence of Lemma \ref{lemma5.1}, one finds from 
(\ref{2.21}) and (\ref{2.23}) that there exists an integer $h_{-1}$ with $0\le h_{-1}\le 4B$ such that
$$\llbracket J_{s+r}(X)\rrbracket \ll M^{Bs-sh_{-1}/3}\llbracket K_{0,B+h_{-1}}^{r,r}(X)\rrbracket .$$
We therefore deduce from (\ref{2.24}) that
\begin{equation}\label{9.1}
X^\Lam \ll X^\del \llbracket J_{s+r}(X)\rrbracket \ll X^\del M^{Bs-sh_{-1}/3}
\llbracket K_{0,B+h_{-1}}^{r,r}(X)\rrbracket .
\end{equation}

\par Next we define sequences $(\kap_n)$, $(h_n)$, $(a_n)$, $(b_n)$, $(c_n)$, $(\psi_n)$ and 
$(\gam_n)$, for $0\le n\le N$, in such a way that
\begin{equation}\label{9.2}
2^R\le \kap_{n-1}\le k^R,\quad 0\le h_{n-1}\le 16k^{2R}b_{n-1},
\end{equation}
and
\begin{equation}\label{9.3}
X^\Lam M^{\Lam \psi_n}\ll X^{c_n\del}M^{-\gam_n}\llbracket K_{a_n,b_n}^{r,r}(X)\rrbracket .
\end{equation}
We note here that the sequences $(a_n)$ and $(b_n)$ are not directly related to our earlier use of these 
letters. Given a fixed choice for the sequences $(a_n)$, $(\kap_n)$ and $(h_n)$, the remaining sequences 
are defined by means of the relations
\begin{align}
b_{n+1}&=\kap_nb_n+h_n,\label{9.4}\\
c_{n+1}&=(s/s_0)^R\kap_n(c_n+1),\label{9.5}\\
\psi_{n+1}&=(s/s_0)^R\kap_n(\psi_n+(1-r/s)b_n),\label{9.6}\\
\gam_{n+1}&=(s/s_0)^R\kap_n\gam_n.\label{9.7}
\end{align}
We put
$$\kap_{-1}=k^R,\quad b_{-1}=1,\quad a_0=0,\quad b_0=B+h_{-1},$$
$$\psi_0=0,\quad c_0=1,\quad \gam_0=\tfrac{1}{3}sh_{-1}-Bs,$$
so that both (\ref{9.2}) and (\ref{9.3}) hold with $n=0$ as a consequence of our initial choice of 
$\kap_{-1}$ and $b_{-1}$, together with (\ref{9.1}). We prove by induction that for each non-negative 
integer $n$ with $n<N$, the sequences $(a_m)_{m=0}^n$, $(\kap_m)_{m=0}^n$ and 
$(h_m)_{m=-1}^n$ may be chosen in such a way that
\begin{align}
1\le b_n\le (20Rk^{2R}\tet)^{-1},\quad \psi_n\ge 0,&\quad \gam_n\ge -sb_n,\quad 
0\le c_n\le (2\del)^{-1}\tet,\label{9.8}\\
0\le a_n\le b_n/\sqrt{k},&\quad (k-r+1)b_n\ge ra_n,\label{9.9}
\end{align}
and so that (\ref{9.2}) and (\ref{9.3}) both hold with $n$ replaced by $n+1$.\par

Let $0\le n<N$, and suppose that (\ref{9.2}) and (\ref{9.3}) both hold for the index $n$. We have already 
shown such to be the case for $n=0$. From (\ref{9.2}) and (\ref{9.4}) we find that 
$b_n\le 4(17k^{2R})^nB$, whence, by invoking (\ref{2.8}), we find that for $0\le n\le N$ one has
$$b_n\le (20Rk^{2R}\tet)^{-1}.$$
It is apparent from (\ref{9.5}) and (\ref{9.6}) that $c_n$ and $\psi_n$ are non-negative for all $n$. 
Observe also that since $s\le s_0$ and $\kap_m\le k^R$, then by iterating (\ref{9.5}) we obtain the 
bound
\begin{equation}\label{9.10}
c_n\le k^{Rn}+k^{R}\Bigl( \frac{k^{Rn}-1}{k^R-1}\Bigr) \le 3k^{Rn}\quad (n\ge 0),
\end{equation}
and by reference to (\ref{2.8}) we see that $c_n\le (2\del)^{-1}\tet$ for $0\le n<N$.\par

In order to bound $\gam_n$, we recall that $s\le s_0$ and iterate the relation (\ref{9.7}) to deduce that
\begin{equation}\label{9.11}
\gam_m=(s/s_0)^{Rm}\kap_0\ldots \kap_{m-1}\gam_0\ge -(s/s_0)^{Rm}\kap_0\ldots \kap_{m-1}Bs.
\end{equation}
In addition, we find from (\ref{9.4}) that for $m\ge 0$ one has $b_{m+1}\ge \kap_mb_m$, so that an 
inductive argument yields the lower bound
$$b_m\ge \kap_0\ldots \kap_{m-1}b_0\ge \kap_0\ldots \kap_{m-1}B.$$
Hence we deduce from (\ref{9.11}) that
$$\gam_m\ge -(s/s_0)^{Rm}sb_m>-sb_m.$$
Assembling this conclusion together with those of the previous paragraph, we have shown that 
(\ref{9.8}) holds for $0\le n\le N$.\par

At this point in the argument, we may suppose that (\ref{9.3}), (\ref{9.8}) and (\ref{9.9}) hold for 
the index $n$. An application of Lemma \ref{lemma8.2} therefore reveals that there exist real 
numbers $\kap_n$, $h_n$ and $a_n$ satisfying the constraints implied by (\ref{9.2}) with $n$ 
replaced by $n+1$, for which the upper bound (\ref{9.3}) holds for some $a_n$ with 
$0\le a_n\le b_n/\sqrt{k}$, also with $n$ replaced by $n+1$. Our hypothesis on $r$, moreover, 
ensures that $(k-r+1)b_{n+1}\ge ra_{n+1}$. Hence (\ref{9.9}) holds also with $n$ replaced by 
$n+1$. This completes the inductive step, so that in particular the upper bound (\ref{9.3}) holds 
for $0\le n\le N$.\par

We now exploit the bound just established. Since we have the upper bound 
$b_N\le 4(17k^{2R})^NB\le (2\tet)^{-1}$, it is a consequence of Lemma \ref{lemma9.1} that
$$\llbracket K_{a_N,b_N}^{r,r}(X)\rrbracket\ll X^{\Lam+\del}.$$
By combining this with (\ref{9.3}) and (\ref{9.11}), we obtain the bound
\begin{equation}\label{9.12}
X^\Lam M^{\Lam\psi_N}\ll X^{\Lam+(c_N+1)\del}
M^{\kap_0\ldots \kap_{N-1}Bs(s/s_0)^{RN}}.
\end{equation}
Meanwhile, an application of (\ref{9.10}) in combination with (\ref{2.8}) shows that 
$X^{(c_N+1)\del }<M$. We therefore deduce from (\ref{9.12}) that
$$\Lam \psi_N\le (s/s_0)^{RN}\kap_0\ldots \kap_{N-1}Bs+1.$$
Notice here that $\kap_n\ge 2^R$ and
$$s/s_0\ge (s_1-1)/s_0\ge 1-2/s_0>\tfrac{1}{2}.$$
Hence
$$1<(s/s_0)^{RN}\kap_0\ldots \kap_{N-1}Bs,$$
so that
\begin{equation}\label{9.13}
\Lam \psi_N\le 2(s/s_0)^{RN}\kap_0\ldots \kap_{N-1}Bs.
\end{equation}
A further application of the lower bound $b_n\ge \kap_0\ldots \kap_{n-1}B$ leads from (\ref{9.6}) and 
the bound $s\le s_0$ to the relation
\begin{align*}
\psi_{n+1}&=(s/s_0)^R\left( \kap_n\psi_n+\kap_n(1-r/s)b_n\right) \\
&\ge (s/s_0)^R\kap_n\psi_n+(s/s_0)^R\kap_n(1-r/s)\kap_0\ldots \kap_{n-1}B\\
&\ge (s/s_0)^R\kap_n\psi_n+(s/s_0)^{R(n+1)}(1-r/s)\kap_0\ldots \kap_nB.
\end{align*}
An inductive argument therefore delivers the lower bound
$$\psi_N\ge N(1-r/s)(s/s_0)^{RN}\kap_0\ldots \kap_{N-1}B.$$
Thus we deduce from (\ref{9.13}) that
$$\Lam\le \frac{2(s/s_0)^{RN}\kap_0\ldots \kap_{N-1}Bs}{N(1-r/s)(s/s_0)^{RN}
\kap_0\ldots \kap_{N-1}B}=\frac{2s(1-r/s)^{-1}}{N}.$$
Since we are at liberty to take $N$ as large as we please in terms of $s$ and $k$, we are forced to 
conclude that $\Lam\le 0$. In view of our opening discussion, this completes the proof of the theorem.
\end{proof}

\section{Strongly diagonal behaviour}
We are now equipped to establish Theorem \ref{theorem1.1} and its corollary. Let $k$ be an integer with 
$k\ge 7$, put $r=k-\lceil 2\sqrt{k}\rceil +2$, and suppose that $s$ is a natural number with 
$s\ge \tfrac{1}{4}(k+1)^2$. We note that \cite[Theorem 1.1]{FW2013} demonstrates that whenever 
$1\le s\le \tfrac{1}{4}(k+1)^2$, then one has $J_{s,k}(X)\ll X^{s+\eps}$, and so we are certainly at 
liberty to assume that $s>\tfrac{1}{4}(k+1)^2$. Write
\begin{equation}\label{10.1}
\Del=\sum_{m=1}^{r-1}\frac{m(r-m)}{s-m-r},
\end{equation}
and define $s_1$ as in the statement of Theorem \ref{theorem9.2}. We aim to show that the real number
$$t_0=kr-\tfrac{1}{2}r(r+1)-\Del$$
satisfies the bounds
$$\max\{ 2r-1,\lceil \tfrac{1}{2}r(r-1)+\Del\rceil\}\le t_0\le s_1.$$
If these bounds be confirmed, then it follows from Theorem \ref{theorem9.2} that whenever
$$1\le s\le t_0+r=kr-\tfrac{1}{2}r(r-1)-\Del,$$
then $J_{s,k}(X)\ll X^{s+\eps}$, and the conclusion of Theorem \ref{theorem1.1} follows.\par

We observe in our first step that, in view of the lower bound $s>\tfrac{1}{4}(k+1)^2$ and our 
hypothesis $k\ge 7$, one has
\begin{align*}
\Del&\le \sum_{m=1}^{k-5}\frac{m(k-4-m)}{s-k+4-m}\le 
\sum_{m=1}^{k-5}\frac{m(k-4-m)}{\tfrac{1}{4}(k^2-6k+37)}\\
&=\frac{2(k-3)(k-4)(k-5)}{3(k^2-6k+37)}<\tfrac{2}{3}(k-6).
\end{align*}
Next, one finds that $t_0\le s_1$ provided only that
$$(kr-\tfrac{1}{2}r(r-1)-\Del-2r)^2\le (kr-\tfrac{1}{2}r(r-1)-\Del)^2-4r(\tfrac{1}{2}r(r-1)+\Del),$$
an inequality that is satisfied provided that
$$4r(kr-\tfrac{1}{2}r(r-1)-\Del)-4r^2\ge 4r(\tfrac{1}{2}r(r-1)+\Del).$$
This last inequality amounts to the constraint $kr-r(r-1)-r\ge 2\Del$. We therefore conclude that 
$t_0\le s_1$ whenever $r(k-r)\ge \tfrac{4}{3}(k-6)$. But when $k\ge 7$, it is easily verified that
$$r(k-r)\ge (k-\lceil 2\sqrt{k}\rceil +2)(\lceil 2\sqrt{k}\rceil -2)\ge \tfrac{4}{3}(k-6)+1,$$
and hence we confirm that $t_0\le s_1$, as desired.\par

Meanwhile, the bound $\Del<\tfrac{2}{3}(k-6)$ ensures that
\begin{align*}
kr-\tfrac{1}{2}r(r+1)-\Del&\ge (k-r)r+\left( \tfrac{1}{2}r(r-1)+\Del\right) -2\Del \\
&> \tfrac{1}{2}r(r-1)+\Del +\left( r(k-r)-\tfrac{4}{3}(k-6)\right) \\
&\ge \tfrac{1}{2}r(r-1)+\Del +1.
\end{align*}
Since $k\ge 7$, moreover, one has
\begin{align*}
kr-\tfrac{1}{2}r(r+1)-\Del &\ge k(r-1)-\tfrac{1}{2}r(r+1)\ge k(r-1)-\tfrac{1}{2}r(k-3)\\
&=\tfrac{1}{2}(k+3)r-k>2r-1.
\end{align*}
Then we have $t_0\ge \max\{2r-1,\lceil \tfrac{1}{2}r(r-1)+\Del\rceil\}$, confirming the final claimed 
bound. In view of our earlier discussion, the proof of Theorem \ref{theorem1.1} is complete.\par

We turn now to the proof of Corollary \ref{corollary1.2}. We suppose again that $k\ge 7$, and we put
$$s=\tfrac{1}{2}k(k+1)-\lceil \tfrac{7}{3}k\rceil .$$
Again defining $\Del$ as in (\ref{10.1}), in which $r=k-\lceil 2\sqrt{k}\rceil +2$, we find that
\begin{align*}
\Del&\le \sum_{m=1}^{k-5}\frac{m(k-4-m)}{\tfrac{1}{2}k(k+1)-(k-4)-(\tfrac{7}{3}k+\tfrac{2}{3})-m}\\
&\le \frac{(k-3)(k-4)(k-5)}{3k^2-23k+50}<\tfrac{1}{3}k.
\end{align*}
Hence we deduce that
\begin{align*}
kr-&\tfrac{1}{2}r(r-1)-\Del\\
&>k(k-\lceil 2\sqrt{k}\rceil +2)-\tfrac{1}{2}(k-\lceil 2\sqrt{k}\rceil +2)
(k-\lceil 2\sqrt{k}\rceil +1)-\tfrac{1}{3}k\\
&=\tfrac{1}{2}k(k+1)-\tfrac{1}{2}(\lceil 2\sqrt{k}\rceil -1)(\lceil 2\sqrt{k}\rceil -2)-\tfrac{1}{3}k\\
&\ge \tfrac{1}{2}k(k+1)-\sqrt{k}(2\sqrt{k}-1)-\tfrac{1}{3}k>s.
\end{align*}
We therefore deduce from Theorem \ref{theorem1.1} that since
$$s<kr-\tfrac{1}{2}r(r-1)-\sum_{m=1}^{r-1}\frac{m(r-m)}{s-r-m},$$
then $J_{s,k}(X)\ll X^{s+\eps}$. This completes the proof of the corollary.

\section{A refinement of the core iteration}
Our goal in this section is to outline an extension of our basic method that permits, for large exponents 
$k$, the refinement of Corollary \ref{corollary1.2} to deliver Theorem \ref{theorem1.3}. The details of 
this extension are complicated enough that we aim for a somewhat abbreviated account, exploiting the 
discussion of \S\S2--10 so as to provide an outline.\par

In \S2 we fixed $r$ to be a parameter satisfying $1\le r\le k-\lceil 2\sqrt{k}\rceil +2$ in order to ensure 
that, as the iteration bounding $K_{a,b}^{r,r}(X)$ in terms of related mean values $K_{a',b'}^{r,r}(X)$ 
proceeds, the exponents $a'$ and $b'$ satisfy the condition $a'\le b'/\sqrt{k}$. The point of the latter 
constraint, in fact, is to ensure that the iteration tree of mean values does not encounter a situation in 
which $b'<a'$, which would be fatal for the method. By relaxing this condition, but pruning the iteration 
tree, we permit the possibility that $a'$ may be much closer to $b'$ in size, with the exponent mean $s_0$ 
increased so that it becomes closer to $\tfrac{1}{2}k(k+1)$.\par

Let $l=\lceil k^{1/3}\rceil$ and put $r=k-l$. We take $s$ to be a natural number satisfying
$$\tfrac{1}{2}k(k+1)-3k\le s+r\le \tfrac{1}{2}k(k+1)-\tfrac{1}{3}k-8k^{2/3}.$$
We take $R$ and $N$ to be natural numbers sufficiently large in terms of $s$ and $k$ in the same manner 
as in \S2, and define $B$, $\tet$ and $\del$ as in (\ref{2.8}). Other aspects of our initial set-up follow 
those of \S2, mutatis mutandis. Lemma \ref{lemma6.1} remains valid provided that $(k-m)b\ge (m+1)a$. 
If instead one has $(k-m)b<(m+1)a$, then one may make use of an alternate inequality based on the 
application of H\"older's inequality. 

\begin{lemma}\label{lemma11.1} Let $m$ be an integer with $0\le m\le r-1$. Suppose that $a$ and $b$ 
are integers with $0\le a<b\le \tet^{-1}$. Then one has
$$K_{a,b}^{m+1,r}(X)\ll \left( J_{s+r}(X/M^a)\right)^{\ome_1}
\left( J_{s+r}(X/M^b)\right)^{\ome_2}
\left( K_{a,b}^{m,r}(X)\right)^{\ome_3},$$
where
$$\ome_1=\frac{s}{(s+r)(s-m)},\quad \ome_2=\frac{r}{(s+r)(s-m)}\quad \text{and}\quad 
\ome_3=\frac{s-m-1}{s-m}.$$
\end{lemma}

\begin{proof} On considering the underlying Diophantine systems, it follows from (\ref{2.12}) and 
(\ref{2.14}) that for some integers $\xi$ and $\eta$, one has
$$K_{a,b}^{m+1,r}(X)\le \oint |\grF_a^m(\bfalp;\xi)^2\grf_a(\bfalp;\xi)^2\grF_b^r(\bfalp;\eta)^2
\grf_b(\bfalp;\eta)^{2s-2m-2}|\d\bfalp .$$
Then by H\"older's inequality, we obtain
$$K_{a,b}^{m+1,r}(X)\le I_1^{\psi_1}I_2^{\psi_2}I_3^{\psi_3}I_4^{\psi_4}I_5^{\psi_5},$$
where
$$I_1=\oint |\grF_a^m(\bfalp;\xi)^4\grf_a(\bfalp;\xi)^{2s+2r-4m}|\d\bfalp ,$$
$$I_2=\oint |\grf_a(\bfalp;\xi)|^{2s+2r}\d\bfalp ,\quad I_3=\oint |\grf_b(\bfalp;\eta)|^{2s+2r}
\d\bfalp ,$$
$$I_4=\oint |\grF_a^m(\bfalp;\xi)^2\grF_b^r(\bfalp;\eta)^2\grf_b(\bfalp;\eta)^{2s-2m}|\d\bfalp ,$$
$$I_5=\oint |\grF_a^m(\bfalp;\xi)^2\grF_b^r(\bfalp;\eta)^4\grf_b(\bfalp;\eta)^{2s-2r-2m}|\d\bfalp ,$$
$$\psi_1=\frac{1}{2s-2m},\quad \psi_2=\frac{s-r}{(s+r)(2s-2m)},\quad \psi_3=\ome_2,$$
$$\psi_4=\frac{s-m-2}{s-m}\quad \text{and}\quad \psi_5=\frac{1}{s-m}.$$
A further consideration of the underlying Diophantine systems reveals that $I_1\le I_2$ and $I_5\le I_4$, 
and thus it follows from Lemma \ref{lemma2.1} and (\ref{2.12}) that
$$K_{a,b}^{m+1,r}(X)\ll (J_{s+r}(X/M^a))^{\psi_1+\psi_2}(J_{s+r}(X/M^b))^{\ome_2}
(K_{a,b}^{m,r}(X))^{\psi_4+\psi_5}.$$
The conclusion of the lemma follows at once.
\end{proof}

Our substitute for Lemma \ref{lemma7.2} must offer a surrogate for the upper bound supplied by Lemma 
\ref{lemma6.1} in those circumstances wherein $m$ and $a$ are both large. In this context, we introduce 
the parameter $u=u(a,b)$ defined by
\begin{equation}\label{11.1}
u(a,b)=rb/(a+b).
\end{equation}
Notice that when $0\le m\le u-1$, one then has
$$(k-m)b-ma>(k-r)b=lb.$$
We recall also the definitions of $\phi_m$ and $\phi^*$ from (\ref{7.1}) and (\ref{7.2}), and also write
$$\varphi_0=\frac{s(r-u)}{(s+r)(s-u)}\quad \text{and}\quad \varphi_1=\frac{r(r-u)}{(s+r)(s-u)}.$$

\begin{lemma}\label{lemma11.2} Suppose that $a$ and $b$ are integers with $0\le a<b\le \tet^{-1}$, 
and define $u(a,b)$ by means of (\ref{11.1}). Then whenever $a'$ is an integer with $a'\ge a$ for which 
$(k-r+1)b\ge ra'$, and $u$ is an integer with $0\le u\le u(a,b)$, one has
$$\llbracket K_{a,b}^{r,r}(X)\rrbracket \ll 
\left( (X/M^a)^{\varphi_0}(X/M^b)^{\phi^*+\varphi_1}\right)^{\Lam+\del}
\prod_{m=0}^{u-1}\llbracket I_{b,b_m}^{r,r}(X)\rrbracket^{\phi_m},$$
where we write $b_m=(k-m)b-ma'$.
\end{lemma}

\begin{proof} By following the argument of the proof of Lemma \ref{lemma7.1}, though substituting 
Lemma \ref{lemma11.1} in place of Lemma \ref{lemma6.1} when $u\le m\le r-1$, we deduce that
\begin{align*}
K_{a,b}^{r,r}(X)\ll &\, \left( J_{s+r}(X/M^b)\right)^{\phi^*}\left( J_{s+r}(X/M^a)\right)^{\varphi_0}
\left( J_{s+r}(X/M^b)\right)^{\varphi_1}\\
&\, \times \prod_{m=0}^{u-1}\left( (M^{b_m-a})^sI_{b,b_m}^{r,r}(X)\right)^{\phi_m}.
\end{align*}
It may be useful here to note that
$$\sum_{m=u}^{r-1}\phi_m=(s-r)\sum_{m=u}^{r-1}(s_{m+1}^{-1}-s_m^{-1})=(s-r)
\Bigl( \frac{1}{s-r}-\frac{1}{s-u}\Bigr)=\frac{r-u}{s-u}.$$
Thus, on recalling the definitions (\ref{2.21}) to (\ref{2.23}), we obtain the conclusion of the lemma, just 
as in the argument delivering Lemma \ref{lemma7.2}.
\end{proof}

We next combine iterated applications of Lemma \ref{lemma11.2} so as to engineer a block multigrade 
process analogous to that delivered by Lemma \ref{lemma7.3}. Before announcing the lemma that 
summarises this process, we pause to modify the notation of \S7. We consider $R$-tuples $\bfh$ just as in 
the preamble to Lemma \ref{lemma7.3}. The sequences $(a_n)=(a_n(\bfm;\bfh))$ and 
$(b_n)=(b_n(\bfm;\bfh))$ are now defined by putting
$$a_0=\lfloor b/l\rfloor\quad \text{and}\quad b_0=b,$$
and then applying the iterative relations (\ref{7.7}). We then define $u_n(\bfm;\bfh)$ by putting
\begin{equation}\label{11.2}
u_n(\bfm;\bfh)=rb_{n-1}/(a_{n-1}+b_{n-1}).
\end{equation}
We emphasise here that, since $a_{n-1}$ and $b_{n-1}$ depend at most on the first $n-1$ coordinates of 
$\bfm$ and $\bfh$, then the same holds for $u_n(\bfm;\bfh)$. We use the notation
$$\prod_{{\bf0}\le \bfm\le \bfu-1}$$
as shorthand for the ordered product
$$\prod_{0\le m_1\le u_1-1}\prod_{0\le m_2\le u_2-1}\ldots \prod_{0\le m_R\le u_R-1}.$$
Finally, we redefine the quantity 
$\Tet_n(\bfm;\bfh)$ for $0\le n\le R$ by putting
\begin{equation}\label{11.3}
\Tet_n(\bfm;\bfh)=(X/M^{b/2})^{-\Lam-\del}\llbracket K_{a_n,b_n}^{r,r}(X)\rrbracket +M^{-3sk^Rb}.
\end{equation}

\begin{lemma}\label{lemma11.3} Suppose that $a$ and $b$ are integers with 
$$0\le a<b\le (16Rk^{2R}\tet)^{-1},$$
and suppose further that $a\le b/l$. Suppose also that $u_1,\ldots ,u_R$ are integers with 
$0\le u_i\le u_i(\bfm;\bfh)$ $(1\le i\le R)$. Then there exists a choice for $\bfh(\bfm)\in [0,r-1]^R$, 
satisfying the condition that $0\le h_n(\bfm)\le 15k^Rb$ $(1\le n\le R)$, and for which one has
$$(X/M^{b/2})^{-\Lam-\del}\llbracket K_{a,b}^{r,r}(X)\rrbracket \ll \prod_{{\bf0}\le \bfm\le \bfu-1}
\Tet_R(\bfm;\bfh)^{\phi_{m_1}\ldots \phi_{m_R}}.$$
\end{lemma}

\begin{proof} One may follow the inductive strategy adopted in the proof of Lemma \ref{lemma7.3}. The 
key difference in the present setting is that Lemma \ref{lemma11.2} contains additional factors in the 
estimate made available for $\llbracket K_{a,b}^{r,r}(X)\rrbracket$. However, it follows from Lemma 
\ref{lemma11.2} that whenever $u\le u(a,b)$, then
$$(X/M^{b/2})^{-\Lam -\del}\llbracket K_{a,b}^{r,r}(X)\rrbracket \ll (M^{(\Lam+\del)b})^\Ome 
\prod_{m=0}^{u-1}\left( (X/M^{b/2})^{-\Lam-\del}
\llbracket I_{b,b_m}^{r,r}(X)\rrbracket \right)^{\phi_m},$$
where
$$\Ome=\tfrac{1}{2}-\tfrac{1}{2}\sum_{m=0}^{u-1}\phi_m-(\phi^*+\varphi_1).$$
But we have
$$\sum_{m=0}^{u-1}\phi_m=(s-r)\sum_{m=0}^{u-1}\left( \frac{1}{s-m-1}-\frac{1}{s-m}\right) 
=\frac{u(s-r)}{s(s-u)},$$
and hence $2s(s-u)(s+r)\Ome$ is equal to
\begin{align*}
-s(s+r)(s-u)&-u(s-r)(s+r)+2r(s-u)(s+r)-2r(r-u)s\\
&=-s^3+rs^2+urs-ur^2=-(s-r)(s^2-ur).
\end{align*}
Since we may suppose that $s>r$, $s>u$ and $s^2>ur$, we find that $\Ome <0$, and hence
$$(X/M^{b/2})^{-\Lam -\del}\llbracket K_{a,b}^{r,r}(X)\rrbracket \ll 
\prod_{m=0}^{u-1}\left( (X/M^{b/2})^{-\Lam-\del}
\llbracket I_{b,b_m}^{r,r}(X)\rrbracket \right)^{\phi_m}.$$
This relation serves as a substitute for Lemma \ref{lemma7.2} in the proof of Lemma \ref{lemma7.3}, 
the proof of which now applies without serious modification. This completes our account of the proof of 
Lemma \ref{lemma11.3}.
\end{proof}

We must now grapple with the somewhat daunting task of developing an analogue of Lemma 
\ref{lemma8.2} in the current setting. We now define the sequences $(\atil_n)=(\atil_n(\bfm))$ and 
$(\btil_n)=(\btil_n(\bfm))$ by putting
$$\atil_0=1/l\quad \text{and}\quad \btil_0=1,$$
and then applying the relations (\ref{8.3}). We define the parameters $\util_i=\util_i(\bfm)$ and 
$\vtil_i=\vtil_i(\bfm)$ for $1\le i\le R$ by putting
\begin{equation}\label{11.4}
\util_i(\bfm)=r\btil_{i-1}/(\atil_{i-1}+\btil_{i-1})\quad \text{and}\quad 
\vtil_i(\bfm)=\lfloor \util_i(\bfm)\rfloor .
\end{equation}
Next, we define the positive real number $s_0$ by means of the relation
\begin{equation}\label{11.5}
s_0^R=s^R\sum_{m_1=0}^{\vtil_1-1}\sum_{m_2=0}^{\vtil_2-1}\ldots \sum_{m_R=0}^{\vtil_R-1}
\phi_{m_1}\phi_{m_2}\ldots \phi_{m_R}\btil_R(\bfm),
\end{equation}
and then put
\begin{equation}\label{11.6}
k_\bfm=\btil_R(\bfm)\quad \text{and}\quad \rho_\bfm=\btil_R(\bfm)(s/s_0)^R\quad 
(\bfm\in [0,r-1]^R).
\end{equation}
We note here that the values of $k_\bfm$ and $\rho_\bfm$ are not relevant when $\bfm$ does not satisfy 
the condition $0\le m_i\le \vtil_i-1$ $(1\le i\le R)$, and so we are indifferent to the values stemming from 
(\ref{11.6}) in such circumstances.\par

Before proceeding further, we pause to relate $u_i$ and $\util_i$.

\begin{lemma}\label{lemma11.4} Suppose that $a\le b/l$ and $0\le h_n(\bfm)\le 16k^Rb$ 
$(1\le n\le R)$. Then one has $0\le \util_n(\bfm)\le u_n(\bfm;\bfh)$ $(1\le n\le R)$.
\end{lemma}

\begin{proof} We prove the asserted inequalities by induction, beginning with the case $n=1$. From 
(\ref{11.4}), we find on the one hand that
$$\util_1=r\btil_0/(\atil_0+\btil_0)=r/(1+1/l),$$
whilst on the other, from (\ref{11.2}), one has
$$u_1=rb_0/(a_0+b_0)=r/(1+a_0/b)\ge r(1+1/l).$$
Thus $\util_1\le u_1$, confirming the inductive hypothesis when $n=1$. Suppose next that $n\ge 1$ and 
$\util_n\le u_n$. One has
\begin{equation}\label{11.7}
\frac{b_{n+1}}{a_{n+1}+b_{n+1}}\ge \frac{\btil_{n+1}}{\atil_{n+1}+\btil_{n+1}}
\end{equation}
if and only if $b_{n+1}/a_{n+1}\ge \btil_{n+1}/\atil_{n+1}$, which is to say that
$$( (k-m_{n+1})b_n-m_{n+1}a_n+h_n)/b_n\ge ((k-m_{n+1})\btil_n-m_{n+1}\atil_n)/\btil_n.$$
This lower bound is satisfied if and only if $m_{n+1}\atil_n/\btil_n\ge (m_{n+1}a_n-h_n)/b_n$. The 
latter is automatically satisfied when either $m_{n+1}=0$ or $h_n\ge m_{n+1}a_n$, and otherwise it is 
equivalent to the upper bound
$$\frac{\btil_n}{\atil_n+\btil_n}\le \frac{b_n}{a_n+b_n-h_n/m_{n+1}}.$$
However, when $m_{n+1}>0$, in view of our hypothesis $\util_n\le u_n$, one has
$$\frac{\btil_n}{\atil_n+\btil_n}\le \frac{b_n}{a_n+b_n}\le \frac{b_n}{a_n+b_n-h_n/m_{n+1}}.$$
Thus we conclude from (\ref{11.2}), (\ref{11.4}) and (\ref{11.7}) that $\util_{n+1}\le u_{n+1}$. This 
confirms the inductive hypothesis, and hence we deduce that $0\le \util_n\le u_n$ for $1\le n\le R$, as 
claimed.
\end{proof}

We next introduce some additional quantities of use in our ultimate application of Lemma \ref{lemma8.1}. 
Define the exponents $\bet_{\bfm}^{(l)}$ and $\gam_\bfm^{(l)}$ just as in the discussion following 
(\ref{8.8}) above. We now put
\begin{equation}\label{11.8}
B_n=\sum_{m_1=0}^{\vtil_1-1}\ldots \sum_{m_n=0}^{\vtil_n-1}\bet_\bfm^{(n)}\btil_n(\bfm)
\quad \text{and}\quad 
A_n=\sum_{m_1=0}^{\vtil_1-1}\ldots \sum_{m_n=0}^{\vtil_n-1}\bet_\bfm^{(n)}\atil_n(\bfm).
\end{equation}
In order to understand these sequences, we introduce some auxiliary sequences $\Atil_n$ and $\Btil_n$ 
as follows. We put
\begin{equation}\label{11.9}
\gra=\tfrac{1}{2}k(k+1)-\tfrac{1}{3}k-3k^{2/3}\quad \text{and}\quad 
\grb=\tfrac{1}{2}k(k-1)+3k^{5/3}.
\end{equation}
The sequence $\Btil_n$ is then defined for $n\ge 1$ by means of the relations
\begin{equation}\label{11.10}
s\Btil_1=\gra-\grb/l,\quad s^2\Btil_2=\gra(\gra-\grb/l)-r\grb,
\end{equation}
and
\begin{equation}\label{11.11}
s^2\Btil_{n+2}=s\gra\Btil_{n+1}-r\grb \Btil_n\quad (n\ge 1).
\end{equation}
Finally, we put
\begin{equation}\label{11.12}
\tet_\pm=\tfrac{1}{2}\left( \gra\pm \sqrt{\gra^2-4r\grb}\right) .
\end{equation}

\begin{lemma}\label{lemma11.5}
When $n\ge 1$, one has $B_n\ge \Btil_n$. In particular, one has the lower bound $B_R\ge \Btil_R$, and 
hence
$$s_0^R\ge \frac{\tet_+^{R+1}-\tet_-^{R+1}}{\tet_+-\tet_-}-\frac{\tet_+\tet_-}{lr}
\left( \frac{\tet_+^R-\tet_-^R}{\tet_+-\tet_-}\right) .$$
\end{lemma}

\begin{proof} We have
\begin{equation}\label{11.13}
B_{n+1}=\sum_{m_1=0}^{\vtil_1-1}\ldots \sum_{m_n=0}^{\vtil_n-1}\bet_\bfm^{(n)}
\sum_{m_{n+1}=0}^{\vtil_{n+1}-1}\phi_{m_{n+1}}\btil_{n+1}(\bfm).
\end{equation}
The innermost sum here is
\begin{equation}\label{11.14}
\sum_{m_{n+1}=0}^{\vtil_{n+1}-1}\phi_{m_{n+1}}\btil_{n+1}(\bfm)=\grB(\vtil_{n+1}) 
\btil_n(\bfm)-\grA(\vtil_{n+1})\atil_n(\bfm),
\end{equation}
where
$$\grB(u)=\sum_{m=0}^{u-1}(k-m)\phi_m\quad \text{and}\quad 
\grA(u)=\sum_{m=0}^{u-1}m\phi_m.$$

\par We observe that when $u$ is a positive integer, one has
\begin{align*}
\grB(u)&=\sum_{m=0}^{u-1}\frac{(k-m)(s-r)}{(s-m)(s-m-1)}\\
&=\sum_{m=0}^{u-1}\left( \frac{(r-m)(k-m)}{s-m}-\frac{(r-m-1)(k-m-1)}{s-m-1}-\frac{r-m-1}{s-m-1}
\right) ,
\end{align*}
whence
\begin{align*}
s\grB(u)=&\, kr-\frac{s(r-u)(k-u)}{s-u}-s\sum_{m=1}^u\frac{r-m}{s-m}\\
=&\, u(k+r-u)-\frac{u(r-u)(k-u)}{s-u}-\tfrac{1}{2}r(r-1)\\
&\, \ \ \ \ \ \ \ \ \ \ \ +\tfrac{1}{2}(r-u-1)(r-u)-\sum_{m=1}^u\frac{m(r-m)}{s-m}.
\end{align*}
Thus we conclude that
\begin{equation}\label{11.15}
s\grB(u)=ku-\tfrac{1}{2}u(u-1)-\frac{u(r-u)(k-u)}{s-u}-\sum_{m=1}^u\frac{m(r-m)}{s-m}.
\end{equation}
One finds in like manner that
\begin{align}
s\grA(u)&=sk\sum_{m=0}^{u-1}\phi_m-s\sum_{m=0}^{u-1}(k-m)\phi_m\notag \\
&=\frac{ku(s-r)}{s-u}-\left( ku-\tfrac{1}{2}u(u-1)-\frac{u(r-u)(k-u)}{s-u}-\sum_{m=1}^u
\frac{m(r-m)}{s-m}\right)\notag \\
&=\tfrac{1}{2}u(u-1)-\frac{u^2(r-u)}{s-u}+\sum_{m=1}^u\frac{m(r-m)}{s-m}.\label{11.16}
\end{align}

Write
$$v=\vtil_{n+1},\quad h=r-v=k-l-v\quad \text{and}\quad \Deltil=\sum_{m=1}^v\frac{m(r-m)}{s-m}.$$
Then we see from (\ref{11.15}) that $s\grB(v)$ is equal to
\begin{align*}
k(k-l-h)&-\tfrac{1}{2}(k-l-h)(k-l-h-1)-\frac{h(l+h)(k-l-h)}{s-k+l+h}-\Deltil\\
&=\tfrac{1}{2}k(k+1)-\tfrac{1}{2}(l+h)(l+h+1)-\frac{h(l+h)(k-l-h)}{s-k+l+h}-\Deltil .
\end{align*}
Also, from (\ref{11.16}), we find that $s\grA(v)$ is equal to
\begin{align*}
\tfrac{1}{2}&(k-l-h)(k-l-h-1)-\frac{h(k-l-h)^2}{s-k+l+h}+\Deltil \\
&=\tfrac{1}{2}k(k-1)-(l+h)k+\tfrac{1}{2}(l+h)(l+h+1)-\frac{h(k-l-h)^2}{s-k+l+h}+\Deltil .
\end{align*}
On substituting these formulae into (\ref{11.14}), we find that
\begin{equation}\label{11.17}
s\sum_{m_{n+1}=0}^{\vtil_{n+1}-1}\phi_{m_{n+1}}
\btil_{n+1}(\bfm)=\gra_0\btil_n(\bfm)-\grb_0\atil_n(\bfm)-\Ups,
\end{equation}
where
\begin{align*}
\gra_0&=\tfrac{1}{2}k(k+1)-\tfrac{1}{2}l(l+1)-\Deltil,\\
\grb_0&=\tfrac{1}{2}k(k-1)-lk+\tfrac{1}{2}l(l+1)+\Deltil ,
\end{align*}
and
\begin{align}
\Ups=&\, (lh+\tfrac{1}{2}h(h+1))(\atil_n+\btil_n)+\frac{h(l+h)(k-l-h)}{s-k+l+h}\btil_n\notag \\
&\, -\left( hk+\frac{h(k-l-h)^2}{s-k+l+h}\right) \atil_n.\label{11.18}
\end{align}

\par We next seek to estimate the expression $\Ups$, this requiring us to obtain an upper bound for $h$. 
We begin with a proof that $\vtil_n(\bfm)\ge k(1-2/l)$ for $1\le n\le R$. Note first that, since we assume 
$k$ to be sufficiently large and $l=\lceil k^{1/3}\rceil$, one finds from (\ref{11.4}) that
$$\util_1=r\btil_0/(\atil_0+\btil_0)=(k-l)/(1+1/l)>k-2k/l+1.$$
Meanwhile, when $n\ge 1$, the condition
$$m_n\le \vtil_n=r\btil_{n-1}/(\atil_{n-1}+\btil_{n-1})$$
that follows from (\ref{11.4}) ensures that
$$\btil_n=k\btil_{n-1}-m_n(\atil_{n-1}+\btil_{n-1})\ge (k-r)\btil_{n-1}=l\btil_{n-1},$$
whence
$$\util_{n+1}=r\btil_n/(\atil_n+\btil_n)=r/(1+\btil_{n-1}/\btil_n)\ge r/(1+1/l)>k-2k/l+1.$$
We therefore deduce that $\vtil_n\ge k(1-2/l)$ for $n\ge 1$, as we had claimed. It follows, in particular, 
that
$$h=k-l-v\le k-l-k(1-2/l)\le 2k/l\le 2k^{2/3}.$$
We use this opportunity also to recall the bounds $s\ge \tfrac{1}{2}k(k+1)-4k$ and $l\le k^{1/3}+1$.\par

Before exploiting the crude bound for $h$ just obtained, we make use of the more opaque though 
stronger bound
$$h=r-\lfloor r\btil_n/(\atil_n+\btil_n)\rfloor \le r\atil_n/(\atil_n+\btil_n)+1.$$
On substituting these bounds into (\ref{11.18}), we see that
\begin{align*}
\Ups&\le (l+\tfrac{1}{2}(h+1))r\atil_n+(l+\tfrac{1}{2}(h+1))(\atil_n+\btil_n)+9k^{1/3}\btil_n\\
&\le 2k^{5/3}\atil_n+2k^{2/3}\btil_n.
\end{align*}
Finally, we observe that
$$\Deltil \le \sum_{m=1}^{k-4}\frac{m(k-3-m)}{s-k+4}\le 
\frac{\frac{1}{6}(k-2)(k-3)(k-4)}{\frac{1}{2}k(k+1)-5k+4}<\tfrac{1}{3}k+1.$$
By combining these estimates with (\ref{11.17}), we arrive at the relation
\begin{align*}
s\sum_{m_{n+1}=0}^{\vtil_{n+1}-1}\phi_{m_{n+1}}
\btil_{n+1}(\bfm)\ge &\, \left( \tfrac{1}{2}k(k+1)-\tfrac{1}{3}k-3k^{2/3}\right)\btil_n(\bfm)\\
&\, -\left( \tfrac{1}{2}k(k-1)+3k^{5/3}\right) \atil_n(\bfm)\\
&=\gra \btil_n(\bfm)-\grb \atil_n(\bfm).
\end{align*}
Then we deduce from (\ref{11.13}) that
\begin{equation}\label{11.19}
sB_{n+1}\ge \gra B_n-\grb A_n\quad (n\ge 1).
\end{equation}
On the other hand, we have
\begin{align*}
s\sum_{m_{n+1}=0}^{\vtil_{n+1}-1}\phi_{m_{n+1}}\atil_{n+1}(\bfm)&=s\btil_n(\bfm)
\sum_{m=0}^{\vtil_{n+1}-1}\phi_m\\
&=\btil_n(\bfm)\frac{\vtil_{n+1}(s-r)}{s-\vtil_{n+1}}\le r\btil_n(\bfm),
\end{align*}
whence
\begin{equation}\label{11.20}
sA_{n+1}\le rB_n \quad (n\ge 1).
\end{equation}

By combining (\ref{11.19}) and (\ref{11.20}), we conclude at this point that
\begin{equation}\label{11.21}
s^2B_{n+2}\ge s\gra B_{n+1}-r\grb B_n\quad (n\ge 1).
\end{equation}
In addition, we have the initial data
$$sB_1\ge \gra \btil_0-\grb \atil_0=\gra-\grb/l,\quad \text{and}\quad sA_1\le r\btil_0=r,$$
and hence
$$s^2B_2\ge s\gra B_1-s\grb A_1\ge \gra(\gra-\grb/l)-\grb r.$$
Our goal is now to extract from the recurrence inequality (\ref{11.21}) a lower bound for $B_n$. Were 
we to be presented with an equation, this would be straightforward, but in present circumstances we must 
work less directly by relating $B_n$ to $\Btil_n$. By reference to (\ref{11.11}), we see that $\Btil_n$ 
satisfies a recurrence equation related to the inequality (\ref{11.19}).\par

Our first observation is that the recurrence formula (\ref{11.11}) has a solution of the shape
$$s^n\Btil_n=\sig_+\tet_+^n+\sig_-\tet_-^n\quad (n\ge 1),$$
where
$$\sig_+\tet_++\sig_-\tet_-=s\Btil_1=\gra-\grb/l$$
and
$$\sig_+\tet_+^2+\sig_-\tet_-^2=s^2\Btil_2=\gra(\gra-\grb/l)-\grb r.$$
Then we deduce that
\begin{equation}\label{11.22}
s^n\Btil_n=\frac{\tet_+^{n+1}-\tet_-^{n+1}}{\tet_+-\tet_-}-\frac{\tet_+\tet_-}{lr}\left( 
\frac{\tet_+^n-\tet_-^n}{\tet_+-\tet_-}\right) .
\end{equation}
Note that since $k$ is large, it follows from (\ref{11.9}) and (\ref{11.12}) that
\begin{align}
\tet_+&=\tfrac{1}{2}k(k+1)-\tfrac{1}{3}k-3k^{2/3}-\grb r/\gra +O((\grb r)^2/\gra^3)\notag \\
&\ge \tfrac{1}{2}k(k-1)-\tfrac{1}{3}k-7k^{2/3},\label{11.23}
\end{align}
and
$$\tet_-\le \grb r/\gra+O((\grb r)^2/\gra^3)\le k+7k^{2/3}.$$
Thus, in particular, we deduce from (\ref{11.22}) that $\Btil_n$ is positive for every natural number $n$.
\par

We now reinterpret the recurrence inequality (\ref{11.21}) as a recurrence equation with variable shifts. 
Let $(g_n)$ be a sequence of positive real numbers. Then for an appropriate choice of this sequence
 $(g_n)$, the recurrence sequence $B_n$ may be interpreted as the solution of the new recurrence 
sequence
\begin{equation}\label{11.24}
s^2B_{n+2}=s\gra B_{n+1}-r\grb B_n+g_n\quad (n\ge 1),
\end{equation}
with the initial data
\begin{equation}\label{11.25}
sB_1=\gra-\grb/l+g_1\quad \text{and}\quad s^2B_2=\gra(\gra-\grb/l+g_1)-\grb r+g_2.
\end{equation}
Define the real numbers $\gtil_n$ by means of the relation
\begin{equation}\label{11.26}
\gtil_n=s^n(B_n-\Btil_n).
\end{equation}
Thus, in particular, it follows from (\ref{11.10}) and (\ref{11.25}) that
\begin{equation}\label{11.27}
\gtil_1=g_1\quad \text{and}\quad \gtil_2=\gra g_1+g_2,
\end{equation}
and from (\ref{11.11}) and (\ref{11.24}) that
\begin{equation}\label{11.28}
\gtil_{n+2}=\gra \gtil_{n+1}-r\grb \gtil_n+g_n \quad (n\ge 1).
\end{equation}
We claim that $\gtil_n\ge 0$, whence also $B_n\ge \Btil_n$,  for each natural number $n$.\par

In order to verify the last claim, we prove by induction that for each natural number $n$, one has
\begin{equation}\label{11.29}
\gtil_{n+1}\ge s\gtil_n \Btil_{n+1}/\Btil_n.
\end{equation}
We first confirm this inductive hypothesis for $n=1$. Here we note that when $\gtil_1=g_1=0$, then the 
trivial lower bound $\gtil_2=g_2\ge 0$ that follows from (\ref{11.27}) suffices to confirm (\ref{11.29}). 
When $g_1\ne 0$, meanwhile, we find from (\ref{11.10}) and (\ref{11.27}) that
\begin{align*}
\gtil_2&=\gtil_1\left( \frac{\gra g_1+g_2}{g_1}\right) =\gtil_1\left( \gra+g_2/g_1\right) \\
&\ge \gtil_1 \left( \frac{\gra(\gra-\grb/l)-\grb r}{\gra-\grb/l}\right) =s\gtil_1\Btil_2/\Btil_1 .
\end{align*}
This confirms the inductive hypothesis when $n=1$.

\par Suppose next that $n\ge 2$ and
$$\gtil_n\ge s\gtil_{n-1}\Btil_n/\Btil_{n-1}.$$
Then, on making use of (\ref{11.11}) and (\ref{11.28}), one deduces that
\begin{align*}
\gtil_{n+1}-s\gtil_n\Btil_{n+1}/\Btil_n&=\left( \gra \gtil_n -r\grb \gtil_{n-1}+g_{n-1}\right) 
-\gtil_n\left( \gra \Btil_n-s^{-1}r\grb \Btil_{n-1}\right) /\Btil_n\\
&\ge (s^{-1}r\grb \Btil_{n-1}/\Btil_n)\left( \gtil_n -s\gtil_{n-1}\Btil_n/\Btil_{n-1}\right) \ge 0.
\end{align*}
This confirms the inductive hypothesis with $n+1$ in place of $n$, and so by applying induction we 
deduce that (\ref{11.29}) holds for every natural number $n$. A particular consequence of the lower 
bound (\ref{11.29}) is that $\gtil_n\ge 0$ for every natural number $n$. Finally, we conclude from 
(\ref{11.26}) that $B_n\ge \Btil_n$ for every natural number $n$.\par

In order to complete the proof of the lemma, we note first from (\ref{11.5}) and (\ref{11.8}) that 
$s_0^R=s^RB_R$. Since $B_R\ge \Btil_R$, the final conclusion of the lemma follows from (\ref{11.22}).
\end{proof}

We next turn to the problem of controlling the behaviour of $a_R$ and $b_R$.

\begin{lemma}\label{lemma11.6}
Suppose that the tuples $\bfm$ and $\bfh(\bfm)$ satisfy the conditions
$$1\le m_n\le \util_n(\bfm)-1\quad \text{and}\quad 0\le h_n(\bfm)\le 15k^Rb\quad (1\le n\le R).$$
Then one has
$$k_\bfm b\le b_R(\bfm;\bfh)\le k_\bfm b+16k^{2R}b,$$
and furthermore
$$a_R(\bfm;\bfh)\le b_R(\bfm;\bfh)/l.$$ 
\end{lemma}

\begin{proof} We again make use of auxiliary recurrence sequences in order to disentangle information 
on recurrence inequalities. Write
$$h_1^*(\bfm)=h_1(\bfm)+m_1(b/l-a_0),$$
and
$$h_n^*(\bfm)=h_n(\bfm)\quad (n>1).$$
Then we find that
$$h_1^*(\bfm)\le h_1(\bfm)+m_1\le h_1(\bfm)+k<16k^Rb,$$
and
$$h_n^*(\bfm)=h_n(\bfm)\le 15k^Rb\quad (n>1).$$
We define the sequences $(a_n^*)=(a_n^*(\bfm;\bfh))$ and $(b_n^*)=(b_n^*(\bfm;\bfh))$ by means 
of the relations
$$a_0^*=b/l\quad \text{and}\quad b_0^*=b,$$
and
$$a_n^*=b_{n-1}^*\quad \text{and}\quad 
b_n^*=(k-m_n)b_{n-1}^*-m_na_{n-1}^*+h_n^*(\bfm)\quad (1\le n\le R).$$
It follows that $a_n=a_n^*$ and $b_n=b_n^*$ for $n\ge 1$. We next put
\begin{equation}\label{11.30}
g_n^*=b_n^*-\btil_n b\quad (n\ge 0),
\end{equation}
and we seek to show in the first instance that $g_n^*\ge 0$ for each $n$. This of course implies in 
particular that $b_R^*\ge \btil_R b=k_\bfm b$.\par

We prove that $g_n^*\ge 0$ for each $n$ by induction. Observe first that by applying (\ref{11.30}) in 
combination with the recurrence relations for $\atil_n$, $\btil_n$, $a_n^*$, $b_n^*$, we obtain
\begin{align*}
g_0^*\btil_1-\btil_0g_1^*=&\, (b_0^*-\btil_0 b)((k-m_1)\btil_0-m_1\atil_0)\\
&\,-\btil_0((k-m_1)(b_0^*-\btil_0 b)-m_1(a_0^*-\atil_0 b)+h_1^*),
\end{align*}
whence $g_0^*\btil_1-\btil_0g_1^*=-h_1^*\le 0$. Moreover, if we assume that 
\begin{equation}\label{11.31}
g_{n-1}^*\btil_n-\btil_{n-1}g_n^*\le 0,
\end{equation}
then we find in like manner that
\begin{align*}
g_n^*\btil_{n+1}-\btil_n g_{n+1}^*=&\, (b_n^*-\btil_n b)((k-m_{n+1})\btil_n-m_{n+1}\atil_n)\\
&\, -\btil_n ((k-m_{n+1})(b_n^*-\btil_n b)-m_{n+1}(a_n^*-\atil_n b)+h_{n+1}^*).
\end{align*}
Thus we obtain the bound
\begin{align*}
g_n^*\btil_{n+1}-\btil_n g_{n+1}^*&\le m_{n+1}(a_n^*\btil_n-b_n^*\atil_n)\\
&=m_{n+1}((g_{n-1}^*+\btil_{n-1}b)\btil_n-(g_n^*+\btil_n b)\btil_{n-1})\\
&=m_{n+1}(g_{n-1}^*\btil_n-g_n^*\btil_{n-1}).
\end{align*}
We thus conclude from (\ref{11.31}) that $g_n^*\btil_{n+1}-\btil_n g_{n+1}^*\le 0$, thereby 
confirming the inductive hypothesis (\ref{11.31}) with $n+1$ in place of $n$. We therefore deduce by 
induction that $g_{n-1}^*\btil_n\le \btil_{n-1}g_n^*$ for every $n$, whence
\begin{equation}\label{11.32}
g_n^*\ge g_{n-1}^*\btil_n/\btil_{n-1}\ge 0\quad (n\ge 1).
\end{equation}
In this way, we therefore conclude that $b_R=b_R^*\ge \btil_R b=k_\bfm b$, as desired.\par

Having confirmed the lower bound on $b_R(\bfm;\bfh)$ claimed in the statement of the lemma, we 
turn our attention next to the upper bound. Observe that the recurrence relations for $\atil_n$, $\btil_n$, 
$a_n^*$, $b_n^*$ lead us from (\ref{11.30}) to the relation
$$g_n^*=(k-m_n)g_{n-1}^*-m_ng_{n-2}^*+h_n^*(\bfm)\quad (n\ge 2).$$ 
We therefore have the trivial upper bound
$$g_n^*\le kg_{n-1}^*+15k^Rb\quad (n\ge 2),$$
and so we deduce by induction that $g_n^*\le 16k^{R+n}b$. Hence, on recalling (\ref{11.30}) once 
again, we obtain the bound 
$$b_R^*\le \btil_R b+16k^{2R}b=k_\bfm b+16k^{2R}b.$$

\par Collecting together the conclusions of the last two paragraphs, we find that
$$k_\bfm b\le b_R(\bfm;\bfh)\le k_\bfm b+16k^{2R}b,$$
thereby confirming the first assertion of the lemma. It remains now only to bound $a_R$ in terms of 
$b_R$. To this end, we observe that for each $n\ge 0$, it follows from (\ref{11.4}) and our hypothesis 
$m_{n+1}\le \util_n(\bfm)-1$ that $m_{n+1}\le r\btil_n/(\atil_n+\btil_n)$, and hence
$$\btil_{n+1}\ge (k-m_{n+1})\btil_n-m_{n+1}\atil_n\ge (k-r)\btil_n=l\btil_n.$$
We therefore deduce from (\ref{11.32}) that $g_{n+1}^*\ge lg_n^*$ for $n\ge 0$. Consequently, 
recalling again the relation (\ref{11.30}), we see that
$$b_R=b_R^*=\btil_Rb+g_R^*\ge l(\btil_{R-1}b+g_{R-1}^*)=lb_{R-1}^*=la_R.$$
We have therefore shown that $a_R\le b_R/l$, as desired. This completes the proof of the lemma. 
\end{proof}

Before initiating our discussion of the next lemma, we observe that in view of (\ref{11.23}) and our 
hypotheses concerning $s$, one has
$$\tet_++r\ge \tfrac{1}{2}k(k-1)-\tfrac{1}{3}k-7k^{2/3}+(k-l)>s+r.$$
Then it follows that, provided $R$ is sufficiently large in terms of $s$ and $k$, one has $s<\tet_+$, and 
so we deduce from Lemma \ref{lemma11.5} that
\begin{equation}\label{11.33}
s<s_0.
\end{equation}

\begin{lemma}\label{lemma11.7}
Suppose that $\Lam\ge 0$, let $a$ and $b$ be integers with $0\le a<b\le (20Rk^{2R}\tet)^{-1}$, and 
suppose further that $a\le b/l$. Suppose in addition that there are real numbers $\psi$, $c$ and $\gam$, 
with
$$0\le c\le (2\del)^{-1}\tet,\quad \gam\ge -sb\quad \text{and}\quad \psi\ge 0,$$
such that
\begin{equation}\label{11.34}
X^\Lam M^{\Lam \psi}\ll X^{c\del}M^{-\gam}\llbracket K_{a,b}^{r,r}(X)\rrbracket .
\end{equation} 
Then for some $\bfm\in [0,r-1]^R$, there is a real number $h$ with $0\le h\le 16k^{2R}b$, and positive 
integers $a'$ and $b'$ with $a'\le b'/l$, such that
\begin{equation}\label{11.35}
X^\Lam M^{\Lam \psi'}\ll X^{c'\del}M^{-\gam'}\llbracket K_{a',b'}^{r,r}(X)\rrbracket ,
\end{equation}
where $\psi'$, $c'$, $\gam'$ and $b'$ are real numbers satisfying the conditions
$$\psi'=\rho_\bfm(\psi+(\tfrac{1}{2}-r/s)b),\quad c'=\rho_\bfm(c+1),\quad \gam'=\rho_\bfm \gam,
\quad b'=k_\bfm b+h.$$
Moreover, the real number $k_\bfm$ satisfies $2^R\le k_\bfm\le k^R$.
\end{lemma}

\begin{proof} We deduce from the postulated bound (\ref{11.34}), together with Lemmata 
\ref{lemma11.3} and \ref{lemma11.4}, that there exists a choice $\bfh=\bfh(\bfm)$ of tuples, with
$$0\le h_n(\bfm)\le 15k^Rb\quad (1\le n\le R),$$
such that
$$X^\Lam M^{\Lam \psi}\ll X^{(c+1)\del}M^{-\gam}(X/M^{b/2})^\Lam 
\prod_{{\bf0}\le \bfm\le \bfutil-1}\Tet_R(\bfm;\bfh)^{\phi_{m_1}\ldots \phi_{m_R}}.$$
Consequently, one has
$$\prod_{{\bf0}\le \bfm\le \bfutil-1}\Tet_R(\bfm;\bfh)^{\phi_{m_1}\ldots \phi_{m_R}}\gg 
X^{-(c+1)\del}M^{\Lam(\psi+b/2)+\gam }.$$
Note that from (\ref{7.1}) one has
$$\sum_{m=0}^{v-1}\phi_m=\frac{v(s-r)}{s(s-v)}\le \frac{r}{s},$$
so that
$$\sum_{m_1=0}^{\vtil_1-1}\ldots \sum_{m_R=0}^{\vtil_R-1}\phi_{m_1}\ldots \phi_{m_R}\le 
(r/s)^R\le r/s.$$
Then we deduce from the definition (\ref{11.3}) of $\Tet_n(\bfm;\bfh)$ that
\begin{align}
\prod_{{\bf0}\le \bfm\le \bfutil-1}\left( X^{-\Lam }\llbracket K_{a_R,b_R}^{r,r}(X)\rrbracket 
+M^{-3sk^Rb}\right)^{\phi_{m_1}\ldots \phi_{m_R}}&\notag \\
\gg X^{-(c+1)\del}M^{\Lam (\psi+(\frac{1}{2}-r/s)b)+\gam}&.\label{11.36}
\end{align}

\par Put
\begin{equation}\label{11.37}
\Ome =B_R=\sum_{m_1=0}^{\vtil_1-1}\ldots 
\sum_{m_R=0}^{\vtil_R-1}\bet_\bfm^{(R)}\btil_R(\bfm),
\end{equation}
so that in view of Lemma \ref{lemma11.5}, one has $\Ome\ge \Btil_R$. Then an application of Lemma 
\ref{lemma8.1} to (\ref{11.36}) yields the relation
\begin{align*}
\sum_{m_1=0}^{\vtil_1-1}\ldots \sum_{m_R=0}^{\vtil_R-1}&\left( X^{-\Lam }
\llbracket K_{a_R,b_R}^{r,r}(X)\rrbracket 
+M^{-3sk^Rb}\right)^{\Ome /\btil_R(\bfm)}\\
&\gg X^{-(c+1)\del}M^{\Lam (\psi+(\frac{1}{2}-r/s)b)+\gam}.
\end{align*}
We see from (\ref{11.5}) and (\ref{11.37}) that $\Ome =(s_0/s)^R$, and so it follows from (\ref{11.6}) 
that $\Ome/\btil_R(\bfm)=1/\rho_\bfm$. Thus we conclude that for some $R$-tuple $\bfm$, one has
\begin{align}
X^{-\Lam }\llbracket K_{a_R,b_R}^{r,r}(X)\rrbracket +M^{-3sk^Rb}&\gg 
X^{-\rho_\bfm (c+1)\del}M^{\Lam \rho_\bfm (\psi+(\frac{1}{2}-r/s)b)+\rho_\bfm \gam}\notag \\
&\gg X^{-c'\del}M^{\Lam \psi'+\gam'}.\label{11.38}
\end{align}

\par We must again handle the removal of the term $M^{-3sk^Rb}$ on the left hand side of 
(\ref{11.38}). Noting that a direct induction leads from (\ref{8.3}) to the upper bound 
$\btil_R(\bfm)\le k^R$, we find by means of (\ref{11.6}) and (\ref{11.33}) that
$$\rho_\bfm=\btil_R(\bfm)(s/s_0)^R<\btil_R(\bfm)\le k^R.$$
We may therefore follow the analysis leading to (\ref{8.16}), mutatis mutandis, to conclude as before 
that
$$X^{-\Lam }\llbracket K_{a_R,b_R}^{r,r}(X)\rrbracket \gg X^{-c'\del}M^{\Lam \psi'+\gam'}.$$

\par We find from Lemma \ref{lemma11.6} that one has $a_R\le b_R/l$, and that 
$$k_\bfm b\le b_R\le k_\bfm b+16k^{2R}b.$$
Then we conclude that there exist integers $b'=b_R$, $a'=a_R$ and $h$ satisfying the conditions
$$0\le h\le 16k^{2R}b,\quad b'=k_\bfm b+h\quad \text{and}\quad a'\le b'/l,$$
for which
$$X^{-\Lam }\llbracket K_{a',b'}^{r,r}(X)\rrbracket \gg X^{-c'\del}M^{\Lam \psi'+\gam'}.$$
The desired conclusion (\ref{11.35}) follows at once, together with its associated conditions.\par

The only task that remains is to confirm the bounds $2^R\le k_\bfm\le k^R$. On the one hand, just as in 
the proof of Lemma \ref{lemma11.6}, one has $\btil_{n+1}\ge l\btil_n$ for each $n\ge 0$. Then an 
inductive argument confirms that $k_\bfm=\btil_R\ge l^R\btil_0\ge 2^R$. On the other hand, an even 
more elementary induction leads from the recurrence relations for $\btil_n$ to the upper bound 
$k_\bfm=\btil_R\le k^R\btil_0=k^R$. This completes the proof of the lemma.
\end{proof}

We now employ the conclusion of the last lemma to bound $J_{s+r}(X)$.

\begin{theorem}\label{theorem11.8}
Suppose that $s$, $k$ and $r$ are natural numbers with $k$ sufficiently large, $r=k-\lceil k^{1/3}\rceil$ 
and $1\le s\le s_1$, where
$$s_1=\tfrac{1}{2}\left( \tfrac{1}{2}k(k+1)-\tfrac{1}{3}k-3k^{2/3}\right) 
\left( 1+\sqrt{1-\frac{4r(\tfrac{1}{2}k(k-1)+3k^{5/3})}{(\tfrac{1}{2}k(k+1)-
\tfrac{1}{3}k-3k^{2/3})^2}}\right) .$$
Then for each $\eps>0$, one has $J_{s+r}(X)\ll X^{s+r+\eps}$.
\end{theorem}

\begin{proof} Defining $s_0$ via (\ref{11.5}) as in the preamble to Lemma \ref{lemma11.4}, one finds 
that there are no serious modifications required in order to apply the argument of the proof of Theorem 
\ref{theorem9.2} to pass from Lemma \ref{lemma11.7} to the conclusion of the theorem, but with $s_0$ 
in place of $s_1=\tet_+$. The argument leading to (\ref{11.33}) shows, however, that whenever 
$s<s_1$, then one has also that $s<s_0$. The conclusion of the theorem is now immediate.
\end{proof}

In order to establish Theorem \ref{theorem1.3}, we have only to note that, on recalling (\ref{11.23}), 
one finds that
\begin{align*}
s_1+r&=\tet_++k-l\ge \tfrac{1}{2}k(k-1)-\tfrac{1}{3}k-7k^{2/3}+(k-k^{1/3})\\
&\ge \tfrac{1}{2}k(k+1)-\tfrac{1}{3}k-8k^{2/3}.
\end{align*}
Thus, whenever
$$1\le s\le \tfrac{1}{2}k(k+1)-\tfrac{1}{3}k-8k^{2/3},$$
one has $J_{s,k}(X)\ll X^{s+\eps}$.

\section{Some consequences of Theorem \ref{theorem1.3}}
We finish this paper by turning our attention to a few consequences of the asymptotically sharpest of our 
conclusions, namely Theorem \ref{theorem1.3}. We begin by proving Theorem \ref{theorem1.5}. 
Suppose that $k$ is a sufficiently large natural number, and put 
$$\grs=\left\lfloor \tfrac{1}{2}k(k+1)-\tfrac{1}{3}k-8k^{2/3}\right\rfloor .$$
In addition, take $s=\tfrac{1}{2}k(k+1)$, and put $\Del=s-\grs$. Then we find that 
\begin{equation}\label{12.1}
\Del\le \tfrac{1}{3}k+8k^{2/3}+1.
\end{equation}
Consequently, on making use of the trivial estimate $f_k(\bfalp;X)=O(X)$, we deduce from (\ref{2.1}) in 
combination with Theorem \ref{theorem1.3} that
\begin{align*}
J_{s,k}(X)&=\oint |f_k(\bfalp;X)|^{2\grs+2\Del}\d\bfalp \ll X^{2\Del}\oint 
|f_k(\bfalp;X)|^{2\grs}\d\bfalp \\
&\ll X^{2\Del}J_{\grs,k}(X)\ll X^{\grs+2\Del+\eps}\ll X^{s+\Del+\eps}.
\end{align*}
The conclusion of Theorem \ref{theorem1.5} is now immediate.\par

We next point out an application of Theorem \ref{theorem1.3} to Tarry's problem. When $h$, $k$ and $s$ 
are positive integers with $h\ge 2$, consider the Diophantine system
\begin{equation}\label{12.2}
\sum_{i=1}^sx_{i1}^j=\sum_{i=1}^sx_{i2}^j=\ldots =\sum_{i=1}^sx_{ih}^j\quad (1\le j\le k).
\end{equation}
Let $W(k,h)$ denote the least natural number $s$ having the property that the simultaneous equations 
(\ref{12.2}) possess an integral solution $\bfx$ with
$$\sum_{i=1}^s x_{iu}^{k+1}\ne \sum_{i=1}^sx_{iv}^{k+1}\quad (1\le u<v\le h).$$

\begin{theorem}\label{theorem12.1} When $h$ and $k$ are natural numbers with $h\ge 2$ and $k$ 
sufficiently large, one has $W(k,h)\le \tfrac{1}{2}k(k+1)+1$.
\end{theorem}

\begin{proof} The argument of the proof of \cite[Theorem 1.3]{Woo2012} shows that $W(k,h)\le s$ 
whenever one can establish the estimate
$$J_{s,k+1}(X)=o(X^{2s-\frac{1}{2}k(k+1)}).$$
But as a consequence of Theorem \ref{theorem1.3}, when $k$ is sufficiently large and
\begin{equation}\label{12.3}
1\le s\le \tfrac{1}{2}(k+1)(k+2)-\tfrac{1}{3}(k+1)-8(k+1)^{2/3},
\end{equation}
one has $J_{s,k+1}(X)\ll X^{s+\eps}$. Observe, however, that 
$X^{s+\eps}=o(X^{2s-\frac{1}{2}k(k+1)})$ provided only that $s<2s-\tfrac{1}{2}k(k+1)$, which is to 
say that $s\ge \tfrac{1}{2}k(k+1)+1$. Moreover, when $k$ is sufficiently large, one finds that the value 
$s=\tfrac{1}{2}k(k+1)+1$ satisfies the condition (\ref{12.3}), and thus we have 
$W(k,h)\le \tfrac{1}{2}k(k+1)+1$. This completes the proof of the theorem.
\end{proof}

The problem of estimating $W(k,h)$ has been investigated extensively by E. M. Wright and L.-K. Hua 
(see \cite{Hua1938}, \cite{Hua1949} and \cite{Wri1948}). L.-K. Hua was able to show that 
$W(k,h)\le k^2(\log k+O(1))$. This bound was improved in \cite[Theorem 1.3]{Woo2012} by means of 
the efficient congruencing method, delivering the upper bound $W(k,h)\le k^2+k-2$. In our most recent 
work on multigrade efficient congruencing, this bound was further improved in 
\cite[Theorem 12.1]{Woo2014} to obtain $W(k,h)\le \tfrac{5}{8}(k+1)^2$ for $k\ge 3$. The bound 
recorded in Theorem \ref{theorem12.1} above achieves the limit of the methods currently employed in 
which an elementary application of Vinogradov's mean value theorem is applied, as described in 
\cite{Woo1996b} and enhanced in \cite[\S12]{Woo2014}. As far as the author is aware, there is no lower 
bound available on $W(k,h)$ superior to $W(k,h)\ge k+1$. There consequently remains a substantial gap 
between our upper and lower bounds in Tarry's problem.\par

Finally, we explore the consequences of Theorem \ref{theorem1.3} in the context of Waring's problem. 
When $s$ and $k$ are natural numbers, let $R_{s,k}(n)$ denote the number of representations of the
 natural number $n$ as the sum of $s$ $k$th powers of positive integers. A formal application of the circle
 method suggests that for $k\ge 3$ and $s\ge k+1$, one should have
\begin{equation}\label{12.4}
R_{s,k}(n)=\frac{\Gam(1+1/k)^s}{\Gam(s/k)}\grS_{s,k}(n)n^{s/k-1}+o(n^{s/k-1}),
\end{equation}
where
$$\grS_{s,k}(n)=\sum_{q=1}^\infty\sum^q_{\substack{a=1\\ (a,q)=1}}\Bigl( q^{-1}\sum_{r=1}^q
e(ar^k/q)\Bigr)^se(-na/q).$$
With suitable congruence conditions imposed on $n$, one has $1\ll \grS_{s,k}(n)\ll n^\eps$, so that 
the conjectured relation (\ref{12.4}) constitutes an honest asymptotic formula. Let $\Gtil(k)$ denote the 
least integer $t$ with the property that, for all $s\ge t$, and all sufficiently large natural numbers $n$, one 
has the asymptotic formula (\ref{12.4}). By combining the conclusion of Theorem \ref{theorem1.3} with
 our recent work concerning the asymptotic formula in Waring's problem \cite{Woo2012b}, and the 
enhancement \cite[Theorem 8.5]{FW2013} of Ford's work \cite{For1995}, we are able to derive 
new upper bounds for $\Gtil(k)$ when $k$ is sufficiently large.\par

\begin{theorem}\label{theorem12.2}
Let $\xi$ denote the real root of the polynomial $6\xi^3+3\xi^2-1$, and put
 $C=(5+6\xi-3\xi^2)/(2+6\xi)$, so that
$$\xi=0.424574\ldots \quad \text{and}\quad C=1.540789\ldots.$$
Then for large values of $k$, one has $\Gtil(k)<Ck^2+O(k^{5/3})$.
\end{theorem}

For comparison, in \cite[Theorem 1.3]{Woo2014}, we derived an analogous bound in which $C$ is 
replaced by the slightly larger number $1.542749\ldots $.

\begin{proof}[The proof of Theorem \ref{theorem12.2}] In general, we write $\Del_{s,k}$ for the least 
real number with the property that, for all $\eps>0$, one has
$$J_{s,k}(X)\ll X^{2s-\frac{1}{2}k(k+1)+\Del_{s,k}+\eps}.$$
Then it follows from \cite[Lemma 10.7]{Woo2014} that one has $\Gtil(k)\le \lfloor u_1(k)\rfloor +1$, 
where
$$u_1(k)=\min_{\substack{1\le t\le k^2-k+1\\ \Del_{t,k}<1}}\underset{2v+w(w-1)<2t}
{\min_{1\le w\le k-1}\min_{v\ge 1}}\, u_0(k,t,v,w),$$
and
$$u_0(k,t,v,w)=2t-\frac{(1-\Del_{t,k})(2t-2v-w(w-1))}{1-\Del_{t,k}+\Del_{v,k}/w}.$$

\par Suppose that $k$ is a large natural number, and let $\bet$ be a positive parameter to be determined 
in due course. We take
$$w=\lfloor \bet k\rfloor\quad \text{and}\quad v=\tfrac{1}{2}k(k+1).$$
Then it follows from Theorem \ref{theorem1.3}, just as in the discussion leading to (\ref{12.1}), that the 
exponent $\Del_{v,k}$ is permissible, where
$$\Del_{v,k}\le \tfrac{1}{3}k+8k^{2/3}+1.$$
Also, by taking $t=k^2-k+1$, one sees from \cite[Corollary 1.2]{Woo2014} that $\Del_{t,k}=0$. Then 
we deduce that
$$u_1(k)\le 2(k^2-k+1)-\frac{2k^2-k(k+1)-(\bet k)^2+O(k)}
{1+\tfrac{1}{3}k/(\bet k)+O(k^{-1/3})}.$$
It follows that
\begin{align*}
u_1(k)/(2k^2)&\le 1-\frac{\bet(\tfrac{1}{2}-\tfrac{1}{2}\bet^2)}{\bet+\frac{1}{3}}+O(k^{-1/3})\\
&=\frac{3\bet^3+3\bet+2}{6\bet+2}+O(k^{-1/3}).\end{align*}
A modest computation confirms that the optimal choice for the parameter $\bet$ is $\xi$, where $\xi$ is 
the real root of the polynomial equation $6\xi^3+3\xi^2-1=0$. With this choice for $\bet$, one finds that
$$u_1(k)\le \left(\frac{5+6\xi-3\xi^2}{2+6\xi}\right)k^2+O(k^{5/3}).$$
The conclusion of Theorem \ref{theorem12.2} is now immediate.
\end{proof}

\bibliographystyle{amsbracket}
\providecommand{\bysame}{\leavevmode\hbox to3em{\hrulefill}\thinspace}

\end{document}